\documentclass[12pt]{amsart}
\usepackage[foot]{amsaddr}
\usepackage[utf8]{inputenc}
\usepackage[margin=1in]{geometry}
\usepackage{amsfonts, amsmath, amssymb, amsthm}
\usepackage{enumitem}
\usepackage[numbers]{natbib}

\usepackage{graphicx}
\usepackage{xcolor}
\usepackage{tikz}
\usetikzlibrary{calc}
\usepackage[numbers]{natbib}
\usepackage[
  colorlinks=true,
  citecolor=blue,
  linkcolor=blue,
  backref=false
]{hyperref}
\usepackage{cleveref}

\theoremstyle{plain}
\newtheorem{theorem}{Theorem}[section]
\newtheorem{lemma}[theorem]{Lemma}
\newtheorem{proposition}[theorem]{Proposition}
\newtheorem{corollary}[theorem]{Corollary}
\newtheorem{conjecture}[theorem]{Conjecture}

\theoremstyle{remark}
\newtheorem{remark}[theorem]{Remark}

\theoremstyle{definition}
\newtheorem{definition}[theorem]{Definition}
\newtheorem{example}[theorem]{Example}
\newtheorem{question}[theorem]{Question}

\AddToHook{env/theorem/begin}
  {\crefalias{theorem}{theorem}}
\AddToHook{env/proposition/begin}
  {\crefalias{theorem}{proposition}}
\AddToHook{env/lemma/begin}
  {\crefalias{theorem}{lemma}}
\AddToHook{env/corollary/begin}
  {\crefalias{theorem}{corollary}}
\AddToHook{env/conjecture/begin}
  {\crefalias{theorem}{conjecture}}
\AddToHook{env/definition/begin}
  {\crefalias{theorem}{definition}}
\AddToHook{env/example/begin}
  {\crefalias{theorem}{example}}
\AddToHook{env/remark/begin}
  {\crefalias{theorem}{remark}}
\AddToHook{env/question/begin}
  {\crefalias{theorem}{question}}

\crefname{theorem}{theorem}{theorems}
\Crefname{theorem}{Theorem}{Theorems}
\crefname{lemma}{lemma}{lemmas}
\Crefname{lemma}{Lemma}{Lemmas}
\crefname{proposition}{proposition}{propositions}
\Crefname{proposition}{Proposition}{Propositions}
\crefname{corollary}{corollary}{corollaries}
\Crefname{corollary}{Corollary}{Corollaries}
\crefname{conjecture}{conjecture}{conjectures}
\Crefname{conjecture}{Conjecture}{Conjectures}
\crefname{remark}{remark}{remarks}
\Crefname{remark}{Remark}{Remarks}
\crefname{definition}{definition}{definitions}
\Crefname{definition}{Definition}{Definitions}
\crefname{example}{example}{examples}
\Crefname{example}{Example}{Examples}
\crefname{question}{question}{questions}
\Crefname{question}{Question}{Questions}

\newcommand{\C}{\mathbb{C}}
\newcommand{\R}{\mathbb{R}}
\newcommand{\bbS}{\mathbb{S}}

\newcommand{\ba}{\mathbf{a}}
\newcommand{\bb}{\mathbf{b}}
\newcommand{\bc}{\mathbf{c}}
\newcommand{\be}{\mathbf{e}}
\newcommand{\bi}{\mathbf{i}}
\newcommand{\bj}{\mathbf{j}}
\newcommand{\bu}{\mathbf{u}}
\newcommand{\bv}{\mathbf{v}}
\newcommand{\bw}{\mathbf{w}}
\newcommand{\bx}{\mathbf{x}}

\newcommand{\cS}{\mathcal{S}}
\newcommand{\cT}{\mathcal{T}}
\newcommand{\cU}{\mathcal{U}}
\newcommand{\T}{\mathcal{T}}

\newcommand{\rO}{\mathrm{O}}
\newcommand{\rSO}{\mathrm{SO}}
\newcommand{\fS}{\mathfrak{S}}

\DeclareMathOperator{\diag}{diag}
\DeclareMathOperator{\rank}{rank}
\DeclareMathOperator{\batorank}{bato-rank}
\DeclareMathOperator{\matmul}{MatMul}

\newcommand{\Latin}{\mathrm{Latin}}
\newcommand{\MaxLatin}{\mathrm{MaxLatin}}

\definecolor{tensorfill}{gray}{0.88}
\definecolor{gridline}{gray}{0.66}
\definecolor{activecell}{gray}{0.88}

\begin{document}

\title[Unifying singular value decompositions of tensors]{Unifying singular value decompositions of tensors\\via aligned orthogonality}
\author{\'{A}lvaro Ribot}
\address{Harvard University}
\email{aribotbarrado@g.harvard.edu}


\begin{abstract}
    We study basis-aligned two-orthogonal (bato) tensors, which can be written as a sum of critical rank-one approximations whose factors are singular vectors of their flattenings. As such, bato tensors admit a Tucker decomposition and a canonical polyadic decomposition that are related to each other. 
    We prove that generic bato decompositions are identifiable and that their truncations  are critical low-bato-rank approximations.
    We also compute the dimension of the set of bato tensors, and identify the irreducible components of its Zariski closure with isotopy classes of maximal partial Latin hyperrectangles.
    Many important tensors are bato, such as determinants, matrix multiplication tensors, and other structure tensors of algebras. 
\end{abstract}
\maketitle

\section{Introduction}

The singular value decomposition (SVD) is arguably the most important tool in applied linear algebra. It says that any matrix $M \in \R^{m \times n}$ admits a factorization of the form $M = U S V^\top$, where $U \in \R^{m \times m}$ and $V \in \R^{n\times n}$ are orthogonal matrices and $S \in \R^{m \times n}$ is a diagonal matrix. Equivalently, the SVD can be expressed as $M = \sum_{i} s_{ii} \bu_i \otimes \bv_i$, where~$\bu_i$ and~$\bv_i$ are the $i$-th columns of $U$ and $V$, respectively.
The relevance of the SVD across sciences comes from the following properties. The decomposition is unique, provided that the nonzero singular values $s_{ii}$ are distinct. The number of summands is the rank of~$M$. And a best low-rank approximation of $M$ is given by truncating the decomposition to its largest singular values \cite{eckart1936approximation}.

There are two natural ways to extend the SVD to higher-order tensors, which come from the two interpretations in the matrix case. First, in the multiplicative core-based expression $(U, V) \cdot S :=  USV^\top$, the matrices $U$ and $V$ perform an orthogonal change of basis along the first and second factors of a core matrix $S$ that is diagonal. The Tucker decomposition~\cite{de2000multilinear} writes a tensor $\T \in \R^{m \times n \times p}$ as $\T = (A, B, C) \cdot \cS$, where the matrices $A \in \R^{m \times m}$, $B \in \R^{n \times n}$ and $C \in \R^{p \times p}$ perform an orthogonal change of basis along the first, second and third factors of a core tensor $\cS \in \R^{m \times n \times p}$ that has orthogonal slices.
Second, in the additive expression~$\sum_{i} s_{ii} \bu_i \otimes \bv_i$, the vectors $\bu_i$ and $\bv_i$ give a short sum of rank-one terms. The canonical polyadic (CP) decomposition writes a tensor~$\T \in \R^{m \times n \times p}$ as~$\T = \sum_{i=1}^r \bu_i \otimes \bv_i \otimes \bw_i$ with $\bu_i \in \R^{m}$, $\bv_i \in \R^{n}$ and $\bw_i \in \R^p$. The minimum number of summands $r$ in a CP decomposition of $\cT$ is its rank. After normalizing the factor vectors and absorbing their norms into scalar coefficients, we may write CP decompositions as $\cT = \sum_{i=1}^r s_i \bu_i \otimes \bv_i \otimes \bw_i$ with $s_i \in \R \setminus \{0\}$ and $\|\bu_i\| = \| \bv_i\| = \| \bw_i\| = 1$.

Every tensor admits a Tucker and a CP decomposition. However, unlike for matrices, for general tensors there is \emph{no} relation between the columns of the matrices $A, B, C$ in a Tucker decomposition and the rank-one terms $\bu_i \otimes \bv_i \otimes \bw_i$ in a CP decomposition. To have such a connection, one must impose additional sparsity assumptions on the core tensor $\cS$. Indeed, a Tucker decomposition can be expressed as $\cT  = (A, B, C) \cdot \cS = \sum_{i,j,k} s_{ijk} \ba_i \otimes \bb_j \otimes \bc_k$ where $\ba_i$, $\bb_j$ and $\bc_k$ are the $i$-th, $j$-th and $k$-th columns of $A$, $B$ and $C$, respectively. Therefore, the number of rank-one terms in such a decomposition exceeds the rank of $\cT$ if $\cS$ is not~sparse. The goal of this paper is to study sparsity patterns of $\cS$ that connect Tucker and CP decompositions while retaining some of the desirable features of the matrix SVD; see \Cref{fig:bato-summary}.

 \begin{figure}
     \centering
     \resizebox{\textwidth}{!}{
 \begin{tikzpicture}[line cap=round,line join=round,font=\fontsize{7}{7}\selectfont,every path/.style={-}]

\def\cell{0.47}
\def\depX{0.25}
\def\depY{0.18}

\newcommand{\tensorbox}[3]{
  \begin{scope}[shift={(#1,#2)}]
    \pgfmathsetmacro{\W}{3*\cell}
    \pgfmathsetmacro{\H}{4*\cell}
    \pgfmathsetmacro{\DX}{2*\depX}
    \pgfmathsetmacro{\DY}{2*\depY}

    \filldraw[fill=tensorfill,draw=black,line width=.68pt]
      (0,0) rectangle (\W,\H);
    \filldraw[fill=tensorfill,draw=black,line width=.68pt]
      (\W,0)--(\W+\DX,\DY)--(\W+\DX,\H+\DY)--(\W,\H)--cycle;
    \filldraw[fill=tensorfill,draw=black,line width=.68pt]
      (0,\H)--(\DX,\H+\DY)--(\W+\DX,\H+\DY)--(\W,\H)--cycle;

    \foreach \j in {1,2}
      \draw[gridline,line width=.23pt]
        ({\j*\cell},0)--({\j*\cell},\H);
    \foreach \i in {1,2,3}
      \draw[gridline,line width=.23pt]
        (0,{\i*\cell})--(\W,{\i*\cell});

    \foreach \j in {1,2}
      \draw[gridline,line width=.23pt]
        ({\j*\cell},\H)--({\j*\cell+\DX},\H+\DY);
    \draw[gridline,line width=.23pt]
      (\depX,\H+\depY)--(\W+\depX,\H+\depY);

    \draw[gridline,line width=.23pt]
      (\W+\depX,\depY)--(\W+\depX,\H+\depY);
    \foreach \i in {1,2,3}
      \draw[gridline,line width=.23pt]
        (\W,{\i*\cell})--(\W+\DX,{\i*\cell+\DY});

    \node[fill=tensorfill,inner sep=.35pt] at (1.5*\cell,2*\cell) {$#3$};
  \end{scope}
}

\newcommand{\squarematrix}[5]{
  \begin{scope}[shift={(#1,#2)}]
    \pgfmathsetmacro{\S}{#3*\cell}
    \filldraw[fill=#5,draw=black,line width=.65pt] (0,0) rectangle (\S,\S);
    \ifnum#3>1
      \foreach \t in {1,...,\numexpr#3-1\relax}{
        \draw[gridline,line width=.23pt] ({\t*\cell},0)--({\t*\cell},\S);
        \draw[gridline,line width=.23pt] (0,{\t*\cell})--(\S,{\t*\cell});
      }
    \fi
    \node[fill=#5,inner sep=.35pt] at (.5*\S,.5*\S) {$#4$};
  \end{scope}
}

\newcommand{\cmatrix}[3]{
  \begin{scope}[shift={(#1,#2)}]
    \pgfmathsetmacro{\W}{2*\cell}
    \pgfmathsetmacro{\DX}{2*\depX}
    \pgfmathsetmacro{\DY}{2*\depY}
    \filldraw[fill=tensorfill,draw=black,line width=.65pt]
      (0,0)--(\W,0)--(\W+\DX,\DY)--(\DX,\DY)--cycle;
    \draw[gridline,line width=.23pt]
      (\cell,0)--(\cell+\DX,\DY);
    \draw[gridline,line width=.23pt]
      (\depX,\depY)--(\W+\depX,\depY);
    \node[fill=tensorfill,inner sep=.35pt]
      at (\cell+\depX,\depY) {$#3$};
  \end{scope}
}

\newcommand{\uvector}[3]{
  \begin{scope}[shift={(#1,#2)}]
    \pgfmathsetmacro{\H}{4*\cell}
    \filldraw[fill=tensorfill,draw=black,line width=.62pt]
      (0,0) rectangle (\cell,\H);
    \foreach \i in {1,2,3}
      \draw[gridline,line width=.23pt]
        (0,{\i*\cell})--(\cell,{\i*\cell});
    \node[fill=tensorfill,inner sep=.35pt] at (.5*\cell,2*\cell) {$#3$};
  \end{scope}
}

\newcommand{\vvector}[3]{
  \begin{scope}[shift={(#1,#2)}]
    \pgfmathsetmacro{\W}{3*\cell}
    \filldraw[fill=tensorfill,draw=black,line width=.62pt]
      (0,0) rectangle (\W,\cell);
    \foreach \j in {1,2}
      \draw[gridline,line width=.23pt]
        ({\j*\cell},0)--({\j*\cell},\cell);
    \node[fill=tensorfill,inner sep=.35pt] at (.5*\W,.5*\cell) {$#3$};
  \end{scope}
}

\newcommand{\wvector}[3]{
  \begin{scope}[shift={(#1,#2)}]
    \pgfmathsetmacro{\DX}{2*\depX}
    \pgfmathsetmacro{\DY}{2*\depY}
    \filldraw[fill=tensorfill,draw=black,line width=.62pt]
      (0,0)--(\cell,0)--(\cell+\DX,\DY)--(\DX,\DY)--cycle;
    \draw[gridline,line width=.23pt]
      (\depX,\depY)--(\cell+\depX,\depY);
    \node[fill=tensorfill,inner sep=.35pt]
      at (.5*\cell+\depX,\depY) {$#3$};
    \draw[black,line width=.62pt]
      (0,0)--(\cell,0)--(\cell+\DX,\DY)--(\DX,\DY)--cycle;
  \end{scope}
}

\newcommand{\rankone}[3]{
  \begin{scope}[shift={(#1,#2)}]
    \uvector{0}{0}{\mathbf u_{#3}}
    \vvector{0.62}{3*\cell}{\mathbf v_{#3}}
    \wvector{0.02}{4*\cell+0.12}{\mathbf w_{\!#3}}
  \end{scope}
}

\newcommand{\voxel}[3]{
  \pgfmathsetmacro{\vx}{(#2-1)*\cell+(#3-1)*\depX}
  \pgfmathsetmacro{\vy}{(4-#1)*\cell+(#3-1)*\depY}
  \filldraw[fill=activecell,draw=black,line width=.42pt]
    (\vx,\vy) rectangle ++(\cell,\cell);
  \filldraw[fill=activecell,draw=black,line width=.42pt]
    (\vx+\cell,\vy)--(\vx+\cell+\depX,\vy+\depY)--
    (\vx+\cell+\depX,\vy+\cell+\depY)--(\vx+\cell,\vy+\cell)--cycle;
  \filldraw[fill=activecell,draw=black,line width=.42pt]
    (\vx,\vy+\cell)--(\vx+\depX,\vy+\cell+\depY)--
    (\vx+\cell+\depX,\vy+\cell+\depY)--(\vx+\cell,\vy+\cell)--cycle;
}

\newcommand{\wirelattice}{
  \pgfmathsetmacro{\W}{3*\cell}
  \pgfmathsetmacro{\H}{4*\cell}
  \pgfmathsetmacro{\DX}{2*\depX}
  \pgfmathsetmacro{\DY}{2*\depY}

  \foreach \k in {0,1,2}{
    \pgfmathsetmacro{\sx}{\k*\depX}
    \pgfmathsetmacro{\sy}{\k*\depY}
    \draw[gridline,line width=.28pt]
      (\sx,\sy) rectangle ++(\W,\H);
    \foreach \j in {1,2}
      \draw[gridline,line width=.22pt]
        (\sx+\j*\cell,\sy)--(\sx+\j*\cell,\sy+\H);
    \foreach \i in {1,2,3}
      \draw[gridline,line width=.22pt]
        (\sx,\sy+\i*\cell)--(\sx+\W,\sy+\i*\cell);
  }
  \foreach \j in {0,1,2,3}{
    \foreach \i in {0,1,2,3,4}{
      \draw[gridline,line width=.22pt]
        ({\j*\cell},{\i*\cell})--
        ({\j*\cell+\DX},{\i*\cell+\DY});
    }
  }
  \draw[gridline,line width=.28pt]
    (0,0)--(\W,0)--(\W+\DX,\DY)--(\W+\DX,\H+\DY)--
    (\DX,\H+\DY)--(0,\H)--cycle;
  \draw[gridline,line width=.28pt]
    (0,0)--(\DX,\DY)
    (\W,0)--(\W+\DX,\DY)
    (0,\H)--(\DX,\H+\DY);
}
\newcommand{\wireoutline}{
  \pgfmathsetmacro{\W}{3*\cell}
  \pgfmathsetmacro{\H}{4*\cell}
  \pgfmathsetmacro{\DX}{2*\depX}
  \pgfmathsetmacro{\DY}{2*\depY}

  \draw[gridline,line width=.22pt] (\cell,0)--(\cell,1.95*\cell);
  \draw[gridline,line width=.22pt] (2*\cell,3.025*\cell)--(2*\cell,\H);
  \draw[gridline,line width=.22pt] (2*\cell,0)--(2*\cell,0.98*\cell);
  \draw[gridline,line width=.22pt] (0,\cell)--(1.98*\cell,\cell);
  \draw[gridline,line width=.22pt] (\W+\depX,3.025*\cell+\depY)--(\W+\depX,\H+\depY);
  \draw[gridline,line width=.22pt] (\cell,\cell)--(\cell+0.95*\depX,\cell+0.95*\depY);
  \draw[gridline,line width=.22pt] (2.03*\cell,3*\cell)--(\W,3*\cell);
  \draw[gridline,line width=.22pt] (\W,3*\cell)--(\W + 0.95*\depX,3*\cell + 0.95*\depY);

  \draw[black,line width=.68pt] (0,0)--(\W,0);
  \draw[black,line width=.68pt] (0,\H)--(\W,\H);
  \draw[black,line width=.68pt] (\DX,\H+\DY)--(\DX+\W,\H+\DY);
  
  \draw[black,line width=.68pt] (0,0)--(0,\H);
  \draw[black,line width=.68pt] (\W,0)--(\W,\H);
  \draw[black,line width=.68pt] (\W+\DX,\DY)--(\W+\DX,\H+\DY);
  
  \draw[black,line width=.68pt] (0,\H)--(\DX,\H+\DY);
  \draw[black,line width=.68pt] (\W,0)--(\W+\DX,\DY);
  \draw[black,line width=.68pt] (\W,\H)--(\W+\DX,\H+\DY);
}

\squarematrix{0.00}{2.80}{4}{A}{tensorfill}
\tensorbox{2.15}{2.80}{\mathcal S}
\squarematrix{4.30}{3.27}{3}{B}{tensorfill}
\cmatrix{2.71}{5.15}{C}
\node (tuckerlabel) at (2.95,2.12) {\normalsize Tucker decomposition};
\node at (2.95,1.72) {\footnotesize core tensor with orthogonal slices};

\node at (6.17,3.74) {$=$};
\tensorbox{6.58}{2.80}{\mathcal T}

\node at (8.82,3.74) {$=$};
\rankone{9.27}{2.80}{1}
\node at (12.07,3.74) {$+\cdots+$};
\rankone{12.91}{2.80}{r}
\node (cplabel) at (12.16,2.12) {\normalsize CP decomposition};
\node at (12.16,1.72) {\footnotesize short sum of rank-one tensors};

\draw[<->,>=latex,line width=.55pt]
  ($(tuckerlabel.east)+(0.78,-0.2)$)--($(cplabel.west)+(-0.78,-0.2)$);
\node at (7.8,1.52) {\footnotesize relation via sparsity};

\tensorbox{4.80}{-1.22}{\mathcal S}
\node at (7.30,-.28) {$=$};

\begin{scope}[shift={(7.90,-1.22)}]
  \wirelattice
  \voxel{4}{1}{2}
  \voxel{3}{2}{2}
  \voxel{2}{3}{2}
  \voxel{3}{3}{1}
  \voxel{2}{2}{1}
  \voxel{1}{1}{1}
  \wireoutline
\end{scope}
\end{tikzpicture}
}
     \caption{Factors in the Tucker decomposition $\cT = (A, B, C) \cdot \cS$ and the CP decomposition $\cT \! \!= \! \sum_{i=1}^r \bu_i \otimes \bv_i \otimes \bw_i$ are related when $\cS$ has certain sparsity.}
     \label{fig:bato-summary}
 \end{figure}

The reader may be wondering why we require that the core tensor $\cS$ in a Tucker decomposition has orthogonal slices instead of requiring that $\cS$ is diagonal. Diagonal tensors have orthogonal slices, but not every tensor admits a Tucker decomposition with a diagonal core tensor due to orthogonality imposed on factors.
Tensors that admit a Tucker decomposition with a diagonal core tensor are called \emph{orthogonally decomposable} (\emph{odeco}) tensors \cite{boralevi2017orthogonal, kolda2001orthogonal, robeva2017singular}. Equivalently, a tensor is odeco if it can be decomposed as $\T = \sum_{i=1}^r \bu_i \otimes \bv_i \otimes \bw_i$ such that for all $i \neq j$ the vectors $\bu_i$, $\bv_i$ and~$\bw_i$ are orthogonal to $\bu_j$, $\bv_j$ and $\bw_j$, respectively. These decompositions satisfy all the good properties of the SVD mentioned above: they are unique, the number of summands is the rank of $\T$, and a best low-rank approximation of $\T$ is given by truncating the decomposition. However, the set of odeco tensors is very low-dimensional compared to the ambient space of all tensors. We study sparsity patterns that are less restrictive than diagonality while retaining some of the aforementioned properties, thus considering higher-dimensional sets of tensors that can be useful in applications. The key concept here is basis-aligned two-orthogonality.

\begin{definition}[Bato]
    A decomposition $\T = \sum_{i=1}^r \bu_i^{(1)} \otimes \cdots \otimes \bu_i^{(d)}$ is \emph{two-orthogonal} if each pair of rank-one summands is orthogonal in at least two factors. If we further require having one orthogonal basis for each factor (that is, for each $k$, the $\bu_i^{(k)}$ and $\bu_j^{(k)}$ are either collinear or orthogonal), we call such a decomposition \emph{basis-aligned two-orthogonal} (\emph{bato}).
\end{definition}

A \emph{bato tensor} is one that admits a bato decomposition. Bato tensors were introduced in~\cite{ribot2026decomposing} while studying the broader class of two-orthogonal tensors to determine which tensors can be decomposed by iteratively subtracting critical rank-one approximations, i.e., critical points of the best rank-one approximation problem. Two-orthogonality also appeared in \cite{vannieuwenhoven2014generic} while studying optimally truncatable decompositions, as in the Eckart-Young theorem \cite{eckart1936approximation}. Here we study the class of bato tensors.

It is worth stressing that the CP decompositions we consider need not be minimal: the number of rank-one summands in a bato decomposition may exceed the tensor rank. We define the \emph{bato rank} of a bato tensor $\T$ as the minimum number of rank-one terms in a bato decomposition of $\T$. 

\begin{example}
    Determinants and matrix multiplication tensors are bato:
    \[
    \det\nolimits_n = \sum_{\sigma \in \fS_n} \operatorname{sgn}(\sigma) \,\be_{\sigma(1)} \otimes \cdots \otimes \be_{\sigma(n)} , \quad \matmul_{n} = \sum_{i,j,k=1}^n E_{ik}^* \otimes E_{kj}^* \otimes E_{ij} \, .
    \]
Their bato rank is $n!$ and $n^3$, respectively. In particular, the standard matrix multiplication algorithm is optimal among bato algorithms.
See \Cref{sec:world} for more details, as well as a discussion of the meaning of bato rank in complexity theory.
\end{example}

The singular value decomposition of a matrix is a bato decomposition, so every matrix is bato. The following results show how bato decompositions unify Tucker and CP for higher-order tensors, and their good properties resemble the SVD for matrices.

\begin{theorem}\label{thm:bato-tucker}
    A bato tensor $\T \in \R^{n_1} \otimes \cdots \otimes \R^{n_d}$ admits a Tucker decomposition
    $$
    \T = (U^{(1)}, \dots, U^{(d)}) \cdot \cS = \sum_{j_1, \dots, j_d}s_{j_1, \dots, j_d } \bu_{j_1}^{(1)} \otimes \cdots \otimes \bu_{j_d}^{(d)}
    $$
    such that every nonzero summand $s_{j_1, \dots, j_d} \bu_{j_1}^{(1)} \otimes \cdots \otimes \bu_{j_d}^{(d)}$ is a critical rank-one approximation of $\T$. Conversely, a tensor with such a decomposition is bato.
\end{theorem}

\Cref{thm:bato-tucker} says that the rank-one summands in a bato decomposition of $\cT$ are critical rank-one approximations and, moreover, the factors of such rank-one summands are left singular vectors of the mode flattenings of $\cT$. This allows us to analyze bato decompositions via the singular value decompositions of flattenings, which leads us to the following result.

\begin{theorem}\label{thm:uniqueness-bato}
    Bato decompositions are generically unique (up to reordering of summands).
\end{theorem}

Canonical polyadic decompositions of general tensors do not have all the properties that one would hope for. The set of tensors of rank at most $r$ is not always Euclidean closed. 
This implies that there may exist tensors of rank greater than $r$ that can be arbitrarily well approximated by rank-$r$ tensors.
By contrast, we show that the sets of tensors of bounded bato rank are Euclidean closed; see \Cref{cor:bato-closed}.

In addition, the best low-rank approximations of a generic tensor (when they exist) cannot be obtained via truncations of a fixed CP decomposition~\cite{vannieuwenhoven2014generic}.
Put differently, the rank-one terms in a best rank-$r$ and rank-$\tilde r$ approximation of a tensor may be different. This implies that, unlike for matrices, the rank-one terms appearing in a low-rank approximation of a tensor may not be informative on their own; they must be considered altogether. Truncations of bato decompositions have better properties than those of arbitrary CP decompositions.

\begin{theorem} \label{thm:truncation-bato-critical}
    Let $\cT \!= \sum_{i=1}^r \bu_{i}^{(1)} \otimes \cdots \otimes \bu_{i}^{(d)}$ be a generic bato decomposition. Then, for all $\tilde r < r$, the truncation $\tilde \cT \!= \sum_{i=1}^{\tilde r} \bu_{i}^{(1)} \otimes \cdots \otimes \bu_{i}^{(d)}$ is a critical bato-rank-$\tilde r$ approximation of~$\cT$.
\end{theorem}
A critical bato-rank-$\tilde r$ approximation of $\cT$ is a critical point of the best bato-rank-$\tilde r$ optimization problem. Note that, unlike for matrices when we keep the largest $\tilde r$ singular values, truncations of bato decompositions are not always optimal; see \Cref{sec:truncations} for more details.

Generic tensors of order greater than two are not bato. We study the algebraic variety defined by the set of bato tensors, i.e., its Zariski closure, which we call the \emph{bato variety}.

\begin{theorem} \label{thm:dimension-bato}
    The bato variety in $ \R^{n_1} \otimes \cdots \otimes \R^{n_d}$ with $n_1 \leq \cdots \leq n_d$ has dimension
    \[
    \prod_{k = 1}^{d-1} n_k + \sum_{k = 1}^{d-1} \binom{n_k}{2} + n_dm - \binom{m+1}{2},
    \]
    where $m = \min\left\{n_d, \prod\nolimits_{k=1}^{d-1}n_k \right\}$.
\end{theorem}

The term $\prod_{k=1}^{d-1} n_k$ above is the maximum bato rank in that tensor space (\Cref{prop:maximal-bato-rank}), while the other terms are the degrees of freedom to pick orthogonal vectors for each factor.

Tucker decompositions of bato tensors give sparse core tensors: distinct support indices differ in at least two coordinates, as in \Cref{fig:bato-summary}. We call such a support a partial Latin hyperrectangle, as they generalize the notion of Latin squares. Despite bato decompositions being generically unique, their Tucker decompositions may use core tensors supported on different partial Latin hyperrectangles. Nevertheless, these Latin hyperrectangles must be isotopic, i.e., one can be obtained from another by permuting the indices in each coordinate.

\begin{theorem} \label{thm:irreducible-components-bato}
    The irreducible components of the bato variety in $ \R^{n_1} \otimes \cdots \otimes \R^{n_d}$ are in bijection with the isotopy classes of maximal partial Latin hyperrectangles in $[n_1] \times \cdots \times [n_d]$.
\end{theorem}

The combinatorics of Latin squares and their generalizations is very rich and far from being well understood. The number of isotopy classes of (complete) Latin squares and hypercubes is only known for small cases \cite{mckay2008census,mckay2007small}. The possible cardinalities of maximal partial Latin hypercubes have been widely studied \cite{britz2015maximal,donovan2024maximal,horak2017maximal,ostergaard2005new}, but to our knowledge, their isotopy classes have not yet been examined.
We provide the number of isotopy classes of maximal partial Latin hyperrectangles and the cardinalities of each class for small cases in Appendix~\ref{sec:small-latin}.

\begin{remark}[For the coordinate-agnostic reader]
    We regard tensors interchangeably as multilinear maps and multidimensional arrays. 
    Our results apply to tensors in $V_1 \otimes \cdots \otimes V_d$ for any finite-dimensional real inner product spaces $V_k$. Inner products are needed, but there is no need to fix orthonormal bases. Indeed, the properties we study are invariant under the natural action of $\rO(V_1) \times \cdots \times \rO(V_d)$ on $V_1 \otimes \cdots \otimes V_d$, where $\rO(V_k)$ denotes the group of linear isometries $A: V_k \to V_k$. That is, these are properties of Cartesian tensors \cite{lim2021tensors}.
\end{remark}

This article is organized as follows. In \Cref{sec:background} we review the background material on tensors and introduce notation. In \Cref{sec:bato-decompositions} we study bato decompositions and prove Theorems~\ref{thm:bato-tucker}, \ref{thm:uniqueness-bato}, and~\ref{thm:truncation-bato-critical}.
In \Cref{sec:bato-variety} we study the set of bato tensors and the bato variety and prove Theorems~\ref{thm:dimension-bato} and~\ref{thm:irreducible-components-bato}.
In \Cref{sec:world} we exhibit examples of bato tensors, including determinants, matrix multiplication tensors, and other structure tensors of algebras.

\section{Background on tensors} \label{sec:background}

Given a positive integer $n$, let $[n]$ denote the set $\{1, 2, \dots, n\}$ and let $\{\be_i \mid i \in [n] \}$ denote the canonical basis in $\R^n$. We equip $\R^{n}$ with the Euclidean inner product $\langle \cdot, \cdot\rangle$, which lets us identify $\R^{n}$ with its dual $(\R^{n})^*$ via~$\bv \mapsto \langle \bv, \cdot \rangle$.
Let $n_1, \dots, n_d$ be positive integers.
We consider tensors in $(\R^{n_1})^* \otimes \cdots \otimes (\R^{n_d})^* \cong \R^{n_1} \otimes \cdots \otimes \R^{n_d}$, that is, the real vector space of multilinear maps~$\T : \R^{n_1} \times \cdots \times \R^{n_d} \to \R$.
The positive integer $d$ is called the \emph{order} of the tensor. Tensors of order one are vectors and tensors of order two are matrices.

\subsection*{CP decomposition and tensor rank}
A tensor $\T$ has \emph{rank one}
if it can be expressed as $\T = \bu^{(1)} \otimes \cdots \otimes \bu^{(d)}$ with~$\bu^{(k)} \in \R^{n_k} \!\setminus \!\{0\}$ for all $k \in [d]$.
A \emph{canonical polyadic} (CP) decomposition of a tensor $\T$ expresses it as a sum of rank-one terms: 
\[
\T = \sum_{i=1}^r \bu_i^{(1)} \otimes \cdots \otimes \bu_i^{(d)}.
\]
The minimum number of terms~$r$ required in a CP decomposition of a tensor is its \emph{rank}.

\subsection*{Inner product and orthonormal basis}
The inner products on each $\R^{n_k}$ induce an inner product $\langle \cdot, \cdot \rangle$ on~$\R^{n_1} \otimes \cdots \otimes \R^{n_d}$ defined on two rank-one tensors as
\[
\langle \bu^{(1)} \otimes \cdots \otimes \bu^{(d)}, \bv^{(1)} \otimes \cdots \otimes \bv^{(d)} \rangle  = \prod_{k=1}^d \langle \bu^{(k)}, \bv^{(k)} \rangle
\]
and extended to the whole space by bilinearity. We use $\langle \cdot, \cdot \rangle$ for all inner products for simplicity in notation, since the one we are referring to
can be understood from the context.
With this inner product, the set $\{\be_{i_1} \otimes \cdots \otimes \be_{i_d} \mid i_{k} \in [n_k] \text{ for all } k \in [d]\}$ is an orthonormal basis of~$\R^{n_1} \otimes \cdots \otimes \R^{n_d}$, where we use the same $\{\be_{i}\}$ to denote the canonical basis of $\R^{n_k}$ for all~$k \in [d]$ again to simplify the notation.
Let $t_{i_1, \dots, i_d}$ denote the coordinates of $\T$ with respect to this basis, so that $\T = \sum_{i_1, \dots, i_d} t_{i_1, \dots, i_d} \be_{i_1} \otimes \cdots \otimes \be_{i_d}$ and $t_{i_1, \dots, i_d} = \T(\be_{i_1}, \dots, \be_{i_d})$. Hence, the inner product between two tensors $\cS$ and $\cT$ can be written as $\langle \cS, \cT \rangle = \sum_{i_1, \dots, i_d} s_{i_1, \dots, i_d} t_{i_1, \dots ,i_d}$.
Using these coordinates, we interchangeably refer to tensors as multilinear maps or multidimensional arrays, i.e., we use the identification $\R^{n_1} \otimes \cdots \otimes \R^{n_d} \cong \R^{n_1 \times \cdots \times n_d}$.

\subsection*{Transforming tensors}
The transpose of a matrix $U \in \R^{m \times n}$ is denoted by $U^\top \in \R^{n \times m}$.
A tuple of matrices $(U^{(1)}, \dots, U^{(d)}) \in \R^{m_1 \times n_1} \times \cdots \times \R^{m_d \times n_d}$ defines a linear map from $\R^{n_1 \times \cdots \times n_d}$ to $\R^{m_1 \times \cdots \times m_d}$ as follows.
Given $\cS \in \R^{n_1 \times \cdots \times n_d}$, the tensor $\T = (U^{(1)}, \dots, U^{(d)}) \cdot \cS \in \R^{m_1 \times \cdots \times m_d}$ is the multilinear map given by 
\[
\T(\bx^{(1)}, \dots, \bx^{(d)}) = \cS({U^{(1)}}^\top \bx^{(1)}, \dots, {U^{(d)}}^\top \bx^{(d)})
\]
for all vectors $\bx^{(k)} \in \R^{m_k}$. In coordinates, for each $(i_1, \dots, i_d) \in [m_1] \times \cdots \times [m_d]$ we have
\[
t_{i_1, \dots, i_d} = \sum_{j_1 = 1}^{n_1} \cdots \sum_{j_d = 1}^{n_d} s_{j_1, \dots, j_d} u^{(1)}_{i_1, j_1} \cdots u^{(d)}_{i_d, j_d}
\]
or, equivalently,
$
\cT = (U^{(1)}, \dots, U^{(d)}) \cdot \cS = \sum_{j_1, \dots, j_d} s_{j_1, \dots, j_d} \bu^{(1)}_{j_1} \otimes \cdots \otimes \bu^{(d)}_{j_d},
$
where $\bu^{(k)}_j$ denotes the $j$-th column of $U^{(k)}$. For example, when $d=2$, we get $(U, V) \cdot \cS = U \cS  V^{\top}$.

\subsection*{Change of basis} For a positive integer $n$, let $\rO(n) = \{ U \in \R^{n \times n} \mid U^\top U = I_n\}$ be the orthogonal group in dimension $n$, where $I_n \in \R^{n \times n}$ is the identity matrix. The group $\rO(n_1) \times \cdots \times \rO(n_d)$ acts on~$\R^{n_1 \times \cdots \times n_d}$ via
\[
\begin{array}{ccc}
    \left( \rO(n_1) \times \cdots \times \rO(n_d) \right)  \times  \R^{n_1 \times \cdots \times n_d} & \to & \R^{n_1 \times \cdots \times n_d} \\[0.5em]
     \left((U^{(1)}, \dots, U^{(d)}), \T \right) & \mapsto & (U^{(1)}, \dots, U^{(d)}) \cdot \T
\end{array}
\]
Geometrically, each $(U^{(1)}, \dots, U^{(d)}) \in \rO(n_1) \times \cdots \times \rO(n_d)$ performs an orthogonal change of basis on~$\R^{n_1 \times \cdots \times n_d}$ sending $\be_{i_1} \otimes \cdots \otimes \be_{i_d}$ to $\bu^{(1)}_{i_1} \otimes \cdots \otimes \bu^{(d)}_{i_d}$ for each
$(i_1, \dots, i_d) \in [n_1] \times \cdots \times [n_d]$.

\subsection*{Flattenings and Tucker decomposition} Given a tensor $\T \in \R^{n_1 \times \cdots \times n_d}$ and $k \in [d]$, its \emph{mode-$k$ flattening} is the matrix $T^{(k)} \in \R^{n_k \times (\prod_{\ell \neq k} n_\ell)}$ obtained by viewing $\T$ as a linear map from $\bigotimes_{\ell \in [d] \setminus \{k\}}\R^{n_\ell}$ to $\R^{n_k}$. A \emph{Tucker decomposition} of $\T$ is
\[
\T = (U^{(1)}, \dots, U^{(d)}) \cdot \cS
\]
where each $U^{(k)}$ is the matrix of left singular vectors of $T^{(k)}$, i.e., $T^{(k)} = U^{(k)} \Sigma^{(k)} {V^{(k)}}^\top$ with $U^{(k)} \in \rO(n_k)$ and $V^{(k)} \in \rO(\prod_{\ell \neq k} n_\ell)$ and $\Sigma^{(k)} \in \R^{n_k \times (\prod_{\ell \neq k} n_\ell)}$ diagonal. The slices of the \emph{core tensor} $\cS$ in a Tucker decomposition are orthogonal to each other:  $\langle \cS_{i_k = \alpha}, \cS_{i_k = \beta}\rangle = 0$ if~$\alpha \neq \beta$, where $\cS_{i_k = \alpha} := (I_{n_1}, \dots, I_{n_{k-1}}, \be_{\alpha}^\top, I_{n_{k+1}}, \dots, I_{n_d}) \cdot \cS \in \R^{n_1 \times \cdots \times n_{k-1} \times n_{k+1}  \times \cdots \times n_d}$. The Tucker decomposition is also known as higher-order singular value decomposition (HOSVD)~\cite{de2000multilinear}.

\subsection*{Singular vector tuple} Let $\bbS^{n-1}=  \{ \bu \in \R^n \mid \| \bu \| = 1\}$ denote the unit sphere in $\R^n$. Given a tensor $\cT \in \R^{n_1 \times \cdots \times n_d}$, a tuple of vectors $(\bu^{(1)}, \dots, \bu^{(d)}) \in \bbS^{n_1-1} \times \cdots \times \bbS^{n_d-1}$ is a~\emph{singular vector tuple} of $\cT$ with \emph{singular value} $\lambda \in \R$ if for all $k \in [d]$ we have
\[
\cT (\bu^{(1)}, \dots, \bu^{(k-1)}, \; \cdot\; , \bu^{(k+1)}, \dots, \bu^{(d)}) = \lambda \bu^{(k)}.
\]
Here, $\cdot$ denotes that we are not contracting with any vector along the $k$-th factor. When $d=2$, we recover the classical notion of singular vector pairs:
$$
\T(\bu^{(1)}, \cdot) := \T^\top \bu^{(1)} = \lambda \bu^{(2)} \quad  \text{and} \quad \T(\cdot, \bu^{(2)}) := \T \bu^{(2)} = \lambda \bu^{(1)}.
$$
Singular vector tuples correspond to \emph{critical rank-one approximations} $\lambda \bu^{(1)} \otimes \cdots \otimes \bu^{(d)}$, i.e., critical points of the squared distance function $\mathrm{dist}^2_\T(\bv^{(1)} \otimes \cdots \otimes \bv^{(d)}) = \|\T - \bv^{(1)} \otimes \cdots \otimes \bv^{(d)} \|^2$; see, for example, \cite[Proposition 2.5]{ribot2026decomposing} for details.

\section{Basis-aligned two-orthogonal decompositions} \label{sec:bato-decompositions}

    A tensor decomposition $\T = \sum_{i=1}^r \bu_i^{(1)} \otimes \cdots \otimes \bu_i^{(d)}$ is \emph{two-orthogonal} if every pair of rank-one terms in the decomposition is orthogonal in at least two factors: for all $i \neq j \in [r]$ there exist $k \neq \ell \in [d]$ such that $\langle \bu_i^{(k)}, \bu_j^{(k)} \rangle = 0$ and $\langle \bu_i^{(\ell)}, \bu_j^{(\ell)} \rangle = 0$. 
Every rank-one term in a two-orthogonal decomposition gives a singular vector tuple of the tensor. Therefore, two-orthogonal decompositions may be obtained by iteratively finding critical rank-one approximations. Two-orthogonality is both necessary and sufficient for this property to hold for any ordering of the summands \cite[Theorem 1.2]{ribot2026decomposing}.

Two-orthogonal decompositions may use multiple orthogonal bases on each factor. In contrast, the matrix SVD and the Tucker decomposition use a single orthogonal basis per factor, and we are interested in retaining this property of the SVD.
A two-orthogonal decomposition $\cT = \sum_{i=1}^{r} \bu^{(1)}_i \otimes \cdots \otimes \bu^{(d)}_i$ is
    \emph{basis-aligned} if $\bu^{(k)}_i$ and $\bu^{(k)}_j$ are collinear or orthogonal for all~$k \in [d]$ and all $i,j \in [r]$. A tensor admitting such a decomposition is called \emph{basis-aligned two-orthogonal} (\emph{bato}).

\begin{example}
    Every matrix is bato, since the singular value decomposition is a bato decomposition. For matrices, the notions of two-orthogonal and bato decompositions are equivalent.
\end{example}

\begin{example}
    The tensor $\be_2 \otimes \be_1 \otimes \be_1 + \be_1 \otimes \be_2 \otimes \be_1 + \be_1 \otimes \be_1 \otimes \be_2 $ is bato. The tensor $\be_1 \otimes \be_1 \otimes \be_1 + (\be_1 + \be_2) \otimes \be_2 \otimes \be_2$ is two-orthogonal but not bato. The tensor $\be_1^{\otimes 3} + (\be_1 + \be_2)^{\otimes 3}$ is not two-orthogonal.
\end{example}

Unlike for matrices, computing all the critical rank-one approximations of a general tensor is infeasible in practice. Even computing the best rank-one approximation is NP-hard in general \cite{hillar2013most}. 
The Tucker decomposition and its variants aim to understand the structure of tensors through their mode-$k$ flattenings, where one can use the spectral machinery of matrices. However, for general tensors $\T \in \R^{n_1 \times \cdots \times n_d}$, there is no relation between the singular vectors of the mode-$k$ flattenings $T^{(k)} \in \R^{n_k \times \prod_{\ell \neq k} n_\ell}$ and the singular vectors of the original tensor $\T$.
Bato tensors can be characterized as those with a structured CP decomposition that is a sum of critical rank-one approximations and may be obtained from the singular vectors of the flattenings. 

\begin{proof}[Proof of \Cref{thm:bato-tucker}]
    We show that the singular vector tuples that appear in a bato decomposition of a tensor can be obtained from the left singular vectors of its flattenings, and that this property characterizes bato tensors.
    
    Suppose that $\T \in \R^{n_1} \otimes \cdots \otimes \R^{n_d}$ admits a Tucker decomposition
    $$
    \T = (U^{(1)}, \dots, U^{(d)}) \cdot \cS = \sum_{j_1 = 1}^{n_1} \cdots \sum_{j_d = 1}^{n_d}s_{j_1, \dots, j_d } \bu_{j_1}^{(1)} \otimes \cdots \otimes \bu_{j_d}^{(d)}
    $$
    such that every nonzero summand $s_{j_1, \dots, j_d} \bu_{j_1}^{(1)} \otimes \cdots \otimes \bu_{j_d}^{(d)}$ is a critical rank-one approximation of $\T$, where $\bu_j^{(k)}$ is the $j$-th column of $U^{(k)} \in \rO(n_k)$. Let $(i_1, \dots, i_d) \in [n_1] \times \cdots \times [n_d]$ such that $s_{i_1, \dots, i_d} \neq 0$, so that $(\bu_{i_1}^{(1)}, \dots, \bu_{i_d}^{(d)})$ is a singular vector tuple of $\cT$ with singular value~$s_{i_1, \dots, i_d}$. Then, for any $k \in [d]$ we have
    \[
    s_{i_1, \dots, i_d} \bu_{i_k}^{(k)} = \cT (\bu_{i_1}^{(1)}, \dots, \bu_{i_{k-1}}^{(k-1)}, \; \cdot \; ,\bu_{i_{k+1}}^{(k+1)}, \dots, \bu_{i_d}^{(d)}) = \sum_{j_k = 1}^{n_k} s_{i_1, \dots, i_{k-1}, j_k, i_{k+1}, \dots, i_d} \bu^{(k)}_{j_k},
    \]
    which implies that $s_{i_1, \dots, i_{k-1}, j_k, i_{k+1}, \dots, i_d} = 0$ for all $j_{k} \in [n_k] \setminus \{ i_k\}$ because the vectors $\{\bu_{j_k}^{(k)}\}_{j_k \in [n_k]}$ are linearly independent for being orthogonal. Hence, every pair of distinct tuples $(i_1, \dots, i_d), (j_1, \dots, j_d)$ such that $s_{i_1, \dots, i_d}, s_{j_1, \dots, j_d} \neq 0$ differ in at least two entries, so $\sum_{j_1, \dots, j_d}s_{j_1, \dots, j_d } \bu_{j_1}^{(1)} \otimes \cdots \otimes \bu_{j_d}^{(d)}$ is a bato decomposition of $\cT$.

    Conversely, suppose that $\cT \in \R^{n_1} \otimes \cdots \otimes \R^{n_d}$ is bato, so there exist orthonormal bases $\{\bu_{j}^{(k)} \mid j \in [n_k]\}$ of $\R^{n_k}$ and scalars $s_{j_1, \dots, j_d}$ such that~$\cT \!= \sum_{j_1, \dots, j_d} s_{j_1, \dots, j_d} \bu_{j_1}^{(1)} \otimes \cdots \otimes \bu_{j_d}^{(d)}$.
    Two-orthogonality requires that each pair of distinct tuples $(i_1, \dots, i_d), (j_1, \dots, j_d)$ such that $s_{i_1, \dots, i_d}, s_{j_1, \dots, j_d} \neq 0$ differ in at least two entries. 
    This implies that any nonzero term $s_{j_1, \dots, j_d} \bu_{j_1}^{(1)} \otimes \cdots \otimes \bu_{j_d}^{(d)}$ is a critical rank-one approximation of $\cT$.
    We need to show that each $\bu_{j_k}^{(k)}$ such that $s_{j_1, \dots, j_d} \neq 0$ for some $j_\ell \in [n_\ell]$ is a left singular vector of the mode-$k$ flattening $T^{(k)} \in \R^{n_k \times (\prod_{\ell \neq k} n_\ell)}$. Fix $k \in [d]$ and consider a vectorization isomorphism $\mathrm{vect} : \bigotimes_{\ell \neq k} \R^{n_\ell} \to \R^{\prod_{\ell \neq k} n_\ell}$ induced by identifying the canonical bases. Then, we have
    \begin{equation*}\label{eq:flattening_bato}
        T^{(k)} = \sum_{j_{k}} \sqrt{\sum_{j_\ell \mid \ell \neq k} s^2_{j_1, \dots, j_d}} \bu^{(k)}_{j_k} \otimes \mathrm{vect} \left(  \frac{1}{\sqrt{\sum_{j_\ell \mid \ell \neq k} s^2_{j_1, \dots, j_d}}}
    \sum_{j_\ell \mid \ell \neq k}s_{j_1, \dots, j_d} \bigotimes_{\ell \neq k} \bu_{j_\ell}^{(\ell)}\right),
    \end{equation*}
    where the terms with zero coefficient $\sum_{j_\ell \mid \ell \neq k} s^2_{j_1, \dots, j_d} = 0$ are omitted. The right-hand side of this expression is a singular value decomposition of $T^{(k)}$. Hence, we get the Tucker decomposition $\cT = (U^{(1)}, \dots, U^{(d)}) \cdot \cS$
    by letting $\bu_1^{(k)}, \dots, \bu_{n_k}^{(k)}$
    be the columns of $U^{(k)} \in \rO(n_k)$ and arranging the scalars~$s_{i_1, \dots, i_d}$ into a tensor $\cS \in \R^{n_1\times \cdots \times n_d}$.
\end{proof}

The proof of \Cref{thm:bato-tucker} unveils that bato decompositions can be expressed as Tucker decompositions with a sparse core tensor $\cS$. The sparsity pattern requires that indices of nonzero entries must differ in at least two entries. The \emph{Hamming distance} on $[n_1] \times \cdots \times [n_d]$ is given by the number of positions in which tuples differ: $\mathrm{Hamming} (\bi, \bj) = |\{ k \in [d] \mid i_k \neq j_k\}|$.

\begin{corollary} \label{cor:bato-support}
    Bato decompositions are of the form $\sum_{\bi \in L} s_{\bi} \bu_{i_1}^{(1)} \otimes \cdots \otimes \bu_{i_d}^{(d)}$ where
    \begin{enumerate}[label=(\roman*)]
    \item $L \subset [n_1] \times \cdots \times [n_d]$ is such that $\mathrm{Hamming}(\bi, \bj) \geq 2$ for all distinct $\bi, \bj \in L$,
    \item $s_\bi \in \R$ for all $\bi \in L$,
    \item $\{\bu_{j}^{(k)} \mid j \in [n_k]\}$ is an orthonormal basis of $\R^{n_k}$ for each $k \in [d]$.
    \end{enumerate}
\end{corollary}

\begin{example}[Odeco tensors] \label{ex:odeco}
An \emph{orthogonally decomposable (odeco)} tensor is of the form
\[
    \T=\sum_{i=1}^r \lambda_i\,
    \bu^{(1)}_i\otimes\cdots\otimes \bu^{(d)}_i,
\]
where, for each mode \(k \in [d]\), the vectors
\(\bu^{(k)}_1,\ldots,\bu^{(k)}_r\) are orthonormal, and $\lambda_i \in \R$.  Every odeco tensor is bato: the support of the core tensor is contained in the diagonal set
\[
    L = \{(i,i,\ldots,i) \mid i\in [r]\}.
\]
\end{example}

\subsection{Partial Latin hyperrectangles}
In this section, we study the support of bato decompositions as introduced in \Cref{cor:bato-support}.
A Latin square is an $n \times n$ array filled with symbols in $[n]$ such that each symbol occurs exactly once in each row and column. A partial Latin square is a partially filled array with no symbol occurring more than once in any row or column. Encoding a filled cell in row $i$, column $j$, containing symbol $k$ by the triple $(i,j,k)$ gives a subset of $[n] \times [n] \times [n]$ in which any two distinct triples differ in at least two coordinates, i.e., they are at Hamming distance at least two, as in \Cref{cor:bato-support}.

\begin{example}
    The following is a Latin square for $n=2$:
    \[
	L = \{ (1, 1, 1), (1,2,2), (2,1,2), (2,2,1)\} \leftrightarrow \begin{array}{|c|c|}
	\hline
	1 & 2 \\
	\hline
	2 & 1\\
	\hline
	\end{array}
	\]
    and it gives bato decompositions of the form
    \[
    s_{111} \bu^{(1)}_{1} \otimes \bu^{(2)}_{1} \otimes \bu^{(3)}_{1} + s_{122} \bu^{(1)}_{1} \otimes \bu^{(2)}_{2} \otimes \bu^{(3)}_{2} + s_{212} \bu^{(1)}_{2} \otimes \bu^{(2)}_{1} \otimes \bu^{(3)}_{2} + s_{221}\bu^{(1)}_{2} \otimes \bu^{(2)}_{2} \otimes \bu^{(3)}_{1}.
    \]
\end{example}
\begin{example}
    For $n=9$, solutions to a Sudoku give Latin squares.
\end{example}

Allowing rectangular and higher-order formats leads to partial Latin hyperrectangles, which have been extensively studied \cite{andres2019colouring, britz2015maximal, donovan2024maximal, ostergaard2005new}.

\begin{definition} \label{def:latin}
	A \emph{partial Latin hyperrectangle} is a set $L \subset [n_1] \times \cdots \times [n_d]$ such that $\mathrm{Hamming}(\bi, \bj) \geq 2$ for all distinct $\bi, \bj \in L$. A partial Latin hyperrectangle is \emph{maximal} if it is not a proper subset of another partial Latin hyperrectangle.
\end{definition}

We write $\Latin(n_1, \dots, n_d)$ for the set of partial Latin hyperrectangles $L \subset [n_1]\times\cdots\times[n_d]$, and we write $\MaxLatin(n_1,\ldots,n_d)$ for the set of maximal elements in $\Latin(n_1, \dots, n_d)$.
With this notation, Latin squares are elements in $\Latin(n,n,n)$ with cardinality $|L| = n^2$.
\begin{remark}
    Not all $L \in \MaxLatin(n, n, n)$ are Latin squares. For example, for $n=2$,
\[
	L = \{ (1, 1, 1), (2,2,2)\} \leftrightarrow \begin{array}{|c|c|}
	\hline
	1 &  \\
	\hline
	 & 2\\
	\hline
	\end{array}
\]
is maximal and has cardinality $2$.
\end{remark}

\begin{example}
	When $d=2$ and $n_1 = n_2 = n$, maximal partial Latin hyperrectangles are of the form  $L = \{ (i, \sigma(i)) \mid i \in [n]\}$ for some permutation $\sigma$ on $[n]$. This gives bato decompositions of the form $U S V^\top$ such that $U, V \in \rO(n)$ and the only nonzero entries of~$S$ are $s_{i, \sigma(i)}$ for $i \in [n]$, so each row and column of $S$ has at most one nonzero entry.
    Taking the permutation $\sigma$ to be the identity corresponds to diagonal matrices $S$, and we recover the singular value decomposition. 
\end{example}

\begin{example}
Elements $L \in \Latin(n, n, n, n)$ with $|L| = n^3$ are Latin cubes. More generally, elements $L \in \Latin(n, \overset{(d)}{\dots}, n)$ with $|L| = n^{d-1}$ are Latin hypercubes.
\end{example}

\begin{remark}
	The elements of $L \in \Latin(n_1, \dots, n_d)$ can be represented by an $n_1 \times \cdots \times n_{d-1}$ partially filled array $A_L$ so that $A_L(i_1, \dots, i_{d-1}) = i_d$ if and only if $(i_1, \dots, i_d) \in L$. Having $\mathrm{Hamming}(\bi, \bj) \geq 2$ for all distinct $\bi, \bj \in L$ is equivalent to having each number in ~$[n_d]$ appear at most once in $A_L$ when we fix all but one coordinate.
\end{remark}

\Cref{cor:bato-support} implies that the Tucker decomposition of a
bato tensor has a core tensor supported on a partial Latin hyperrectangle, which leads to the following definition.

\begin{definition}[$L$-bato tensors]
    Given $L \in \Latin(n_1, \dots, n_d)$, the set of $L$-bato tensors in $\R^{n_1 \times \cdots \times n_d}$ is
    \[
    X_L = \left\{(U^{(1)}, \dots, U^{(d)}) \cdot \cS \mid U^{(k)} \in \rO(n_k), \cS \in \R^{n_1 \times \cdots \times n_d}, s_{\bi} = 0 \text{ for all } \bi \notin L \right\}.
    \]
    The set of bato tensors in $\R^{n_1 \times \cdots \times n_d}$ is
    \[
    X = \bigcup_{L \in \Latin(n_1, \dots, n_d)} X_L.
    \]
\end{definition}

\begin{remark}[Relation to free tensors]
The support condition in \Cref{def:latin} is known as \emph{freeness} in algebraic complexity \cite{conner2021towards}. In our terminology, a tensor is free if it is supported on a partial Latin hyperrectangle after invertible changes of basis. The distinction is that membership in \(X\) requires that such a support be obtained by orthogonal changes of basis.
\end{remark}

\subsection{Generic uniqueness} Given any matrix decomposition $M = \sum_{i=1}^r \ba_i \otimes \bb_i =  A B^\top$ with $A \in \R^{m \times r}$ and $B \in \R^{n \times r}$ we also get the decomposition $M = (AU) (BU)^\top$ for any $U \in \rO(r)$. In order to obtain a unique decomposition, we may require the decomposition to be bato, which leads to the singular value decomposition. The SVD is generically unique (in this context, generic means that the singular values are distinct).
Bato decompositions of higher-order tensors are not always unique:
\begin{equation}\label{eq:non-unique-bato}
    \be_1 \otimes \be_1 \otimes \be_1 + \be_1 \otimes \be_2 \otimes \be_2 + \be_2 \otimes \be_1 \otimes \be_2 + \be_2 \otimes \be_2 \otimes \be_1 = \frac{1}{2}(\be_1 + \be_2)^{\otimes 3} + \frac{1}{2}(\be_1 - \be_2)^{\otimes 3}.
\end{equation}
However, as for matrices, non-unique bato decompositions are special.
Next, we show that bato decompositions are generically unique.
That is, given any $L \in \Latin(n_1, \dots, n_d)$, a generic $L$-bato tensor $\cT \in X_L$ has a unique bato decomposition.
Generic means that the set of $L$-bato tensors without a unique bato decomposition is contained in a proper algebraic subset of~$X_L$.
It is important to consider all possible $L \in \Latin(n_1, \dots, n_d)$ as they lead to different irreducible components of the bato variety (see \Cref{thm:irreducible-components-bato}). Therefore, when we say that bato decompositions are generically unique, we mean that the set of bato tensors with a unique bato decomposition is a Zariski open dense set of the bato variety.

\begin{proof}[Proof of \Cref{thm:uniqueness-bato}]
We show that bato decompositions are generically unique using that the matrix SVD is generically unique and applying this to the flattenings of our tensors.
    Consider any partial Latin hyperrectangle $L \subset [n_1] \times \cdots \times [n_d]$ and, for each $k \in [d]$, let $\{\bu_{i_k}^{(k)} \mid i_k \in [n_k]\}$ be any orthonormal basis of $\R^{n_k}$. Let $s_\bi \in \R$ for $\bi \in L$ be generic and consider the bato decomposition $\cT = \sum_{\bi \in L} s_\bi \bu_{i_1}^{(1)} \otimes \cdots \otimes \bu_{i_d}^{(d)}$. 
    Using \Cref{thm:bato-tucker}, a sufficient condition for the singular value decompositions of the flattenings to be unique is that
    \begin{equation}\label{eq:genericity-uniqueness}
    \sum_{\bi \in L \mid i_k =a} s_{\bi}^2 \neq \sum_{\bi \in L \mid i_k =b} s_{\bi}^2
    \end{equation}
    for all $k \in [d]$ and all distinct $a, b \in [n_k]$ such that $\{\bi \in L \mid i_k = a\} \neq \varnothing$.
    These conditions hold for generic scalars $s_\bi$ with $\bi \in L$.
    Once we know that the orthonormal vectors $\bu_{j}^{(k)}$ can be uniquely recovered (up to sign flip and permutation) from the flattenings of $\T$, the scalars $s_\bi$ can also be uniquely recovered because $s_\bi = \T(\bu_{i_1}^{(1)}, \dots, \bu_{i_d}^{(d)})$ for all $\bi \in [n_1] \times \cdots \times [n_d]$. The singular vectors corresponding to zero singular values may not be uniquely determined, but they do not affect the bato decomposition. The sign flip and permutation of orthogonal vectors does not affect the bato decomposition either.
\end{proof}

\begin{remark}
    The conditions for a bato decomposition $\sum_{\bi \in L}s_{\bi} \bu_{i_1}^{(1)} \otimes \cdots \otimes \bu_{i_d}^{(d)}$ to
    be generic depend only on the scalars $s_\bi$, not on the orthonormal bases $\{\bu_i^{(k)} \mid i \in [n_k]\}$.
\end{remark}

\subsection{Bato rank}
Equation \eqref{eq:non-unique-bato} shows that a bato tensor can admit two bato decompositions with different number of summands. Recall that the number of rank-one summands in a bato decomposition may exceed the tensor rank, i.e., bato decompositions may not be minimal CP decompositions.
This leads us to the following definition.
\begin{definition}
    The \emph{bato rank} of a tensor $\T$, denoted $\batorank(\cT)$, is the minimum number of summands in a bato decomposition of $\T$. If $\cT$ is not bato, $\batorank(\cT) := \infty$.
\end{definition}

The rank and bato rank of a matrix are the same, by the SVD. 
For higher-order tensors, the two notions coincide if the rank is small enough.
Given a bato tensor $\cT \in \R^{n_1 \times \cdots \times n_d}$, one has $\rank(\cT) \leq \batorank(\cT)$, and the inequality can be strict. We suspect that equality holds if $\cT$ is concise and the bato rank of $\cT$ is not greater than the generic rank in $\C^{n_1\times \cdots \times n_d}$; see \cite[\S 2.2]{ribot2026decomposing} for some evidence.

The problem of computing the bato rank can be phrased as follows. Given a bato tensor $\cT \in \R^{n_1 \times \dots \times n_d}$, consider all tensors of the form $\cS = (Q_1, \dots, Q_d) \cdot \cT$ with $Q_k \in \rO(n_k)$ for all~$k$. Then, for those $\cS$ supported on a partial Latin hyperrectangle, choose the sparsest one: the one with fewer nonzero entries. Equivalently, one has
    \[
    \batorank(\cT) = \min\{|L| \mid L \in \Latin(n_1, \dots, n_d), \cT \in X_L \}.
    \]

\begin{remark}
    For tensors of order $d \geq 3$, the notion of bato rank does not coincide with a classical notion of $X$-rank defined by taking secants of a given variety \cite{landsberg2011tensors}.
\end{remark}

Canonical polyadic decompositions of tensors of order greater than two have better identifiability properties than their matrix counterparts: they are unique for small ranks \cite{KRUSKAL197795, chiantini2012generic}. \Cref{thm:uniqueness-bato} implies that imposing the bato structure on CP decompositions leads to generic identifiability regardless of the bato rank.

\begin{corollary}
    For any partial Latin hyperrectangle $L \subset [n_1] \times \!\cdots \!\times [n_d]$, a bato decomposition $\T = \sum_{\bi \in L} s_\bi \bu_{i_1}^{(1)} \otimes \cdots \otimes \bu_{i_d}^{(d)}$ with generic $s_\bi$ is unique, so the bato rank of $\cT$ is~$|L|$.
\end{corollary}

At this point, it is natural to ask how many summands a bato decomposition can have. The maximum ranks of spaces of tensors of order greater than two are not known in general. We determine the maximum bato rank.

\begin{proposition} \label{prop:maximal-bato-rank}
    The maximum bato rank in $\R^{n_1 \times \cdots \times n_d}$ with $n_1 \leq \cdots \leq n_d$ is $\prod_{k=1}^{d-1}n_k$.
\end{proposition}

\begin{proof}
    First, we show that the bato rank cannot exceed $\prod_{k=1}^{d-1} n_k$.
    Given a partial Latin hyperrectangle $L \subset [n_1] \times \cdots \times [n_d]$, distinct tuples in $L$ must differ in at least one of the first $d-1$ coordinates, so $|L|  \leq  \left| [n_1] \times \cdots \times [n_{d-1}] \right| = \prod_{k=1}^{d-1} n_k$. Next, we show that this upper bound is achieved.
    Consider the set 
    \[
    L = \left\{ (i_1, \dots, i_d) \in [n_1] \times \cdots \times [n_d]  \mid i_d \equiv \sum\nolimits_{k=1}^{d-1} i_k \pmod{n_d} \right\}.
    \]
    It follows that $|L| = \prod_{k=1}^{d-1}n_k$ and $\mathrm{Hamming}(\bi, \bj) \geq 2$ for all distinct $\bi, \bj \in L$, so the bato rank of $\T = \sum_{\bi \in L} s_{\bi} \be_{i_1} \otimes \cdots \otimes \be_{i_d}$ with generic scalars $s_{\bi}$ is $\prod_{k=1}^{d-1}n_k$, by \Cref{thm:uniqueness-bato}.
\end{proof}

Additivity under direct sums and multiplicativity under tensor/Kronecker products are fundamental questions for tensor rank and play a central role in algebraic complexity theory. Since these identities fail for ordinary tensor rank, it is natural to ask whether the extra structure imposed by batoness leads to better behaviors. We show that, generically, it does.

Let $\cS= \sum_i \bu_i^{(1)} \otimes \cdots \otimes \bu_i^{(d)}\in \R^{m_1} \otimes \cdots \otimes \R^{m_d}$ and~$\cT = \sum_j \bv_j^{(1)} \otimes \cdots \otimes \bv_j^{(d)} \in \R^{n_1} \otimes \cdots \otimes \R^{n_d}$. Their \emph{direct sum} is
\[
\cS \oplus \cT = \sum\nolimits_i \bu_i^{(1)} \otimes \cdots \otimes \bu_i^{(d)} + \sum\nolimits_j \bv_j^{(1)} \otimes \cdots \otimes \bv_j^{(d)} \in (\R^{m_1} \oplus \R^{n_1}) \otimes \cdots \otimes (\R^{m_d} \oplus \R^{n_d}),
\]
which we think of as an element in $\R^{(m_1 + n_1) \times \cdots \times (m_d + n_d)}$. Their \emph{tensor product} is
\[
\cS \otimes \cT = \sum\nolimits_{i,j} \bu_i^{(1)} \otimes \cdots \otimes \bu_i^{(d)} \otimes \bv_j^{(1)} \otimes \cdots \otimes \bv_j^{(d)} \in \R^{m_1} \otimes \cdots \otimes \R^{m_d} \otimes \R^{n_1} \otimes \cdots \otimes \R^{n_d},
\]
which we think of as an element in $\R^{m_1 \times \cdots \times m_d \times n_1 \times \cdots \times n_d}$. Their \emph{Kronecker product} is
\[
\cS \boxtimes \cT = \sum\nolimits_{i,j} (\bu_i^{(1)} \otimes \bv_j^{(1)}) \otimes \cdots \otimes (\bu_i^{(d)} \otimes \bv_j^{(d)}) \in (\R^{m_1} \otimes \R^{n_1}) \otimes \cdots \otimes (\R^{m_d} \otimes \R^{n_d}),
\]
which we think of as an element of $\R^{m_1n_1 \times \cdots \times m_dn_d}$. Notice the difference between Kronecker and tensor products. For example, the Kronecker product of two matrices is another matrix, while the tensor product of two matrices yields an order-four tensor. By definition, one has
\begin{align}
    &\rank(\cS \oplus \cT) \leq \rank(\cS) + \rank(\cT), \label{eq:rank-subadditive} \\
    &\rank(\cS \boxtimes \cT) \leq \rank(\cS \otimes \cT) \leq \rank(\cS) \cdot \rank(\cT). \label{eq:rank-submultiplicative}
\end{align}

When $\cS$ and $\cT$ are matrices, both \eqref{eq:rank-subadditive} and \eqref{eq:rank-submultiplicative} hold with equality. Strassen conjectured in 1973 that \eqref{eq:rank-subadditive} holds with equality for all higher-order tensors~\cite{strassen1973vermeidung}, but his conjecture was disproved in 2019 \cite{shitov2019counterexamples}. The inequalities in \eqref{eq:rank-submultiplicative} can also be strict \cite{christandl2018tensor}. Understanding when these inequalities are satisfied with equality is a central problem in algebraic complexity theory, e.g., related to the complexity of matrix multiplication \cite{Landsberg_2017}.

The known examples for which \eqref{eq:rank-subadditive} or \eqref{eq:rank-submultiplicative} is strict are special.
To our knowledge, it is not known whether \eqref{eq:rank-subadditive} or \eqref{eq:rank-submultiplicative} are generically satisfied with equality. The analogous inequalities for bato rank are equalities generically.

\begin{proposition}\label{prop:bato-rank-additivie-multiplicative}
    Let $d \geq 2$,  $L \in \Latin(m_1, \dots, m_d)$ and $\tilde{L} \in \Latin(n_1, \dots, n_d)$. Then, for a generic pair of tensors $(\cS, \cT ) \in X_{L} \times X_{\tilde L} \subseteq \R^{m_1 \times \cdots \times m_d} \times \R^{n_1 \times \cdots \times n_d}$ one has
    \begin{align*}
        \batorank(\cS \oplus \cT) & = \batorank(\cS) + \batorank(\cT), \\
        \batorank(\cS \otimes \cT) & = \batorank(\cS) \cdot \batorank(\cT), \\
        \batorank(\cS \boxtimes \cT) & = \batorank(\cS) \cdot \batorank(\cT).
    \end{align*}
    The tensor-product statement remains valid when the two tensors have different orders.
\end{proposition}

\begin{proof}
    First, we show that these operations preserve the property of being bato.
    The direct sum of bato tensors is bato: $\cS \oplus \cT \in X_{L \sqcup \tilde{L}}$ where $L \sqcup \tilde L$ can be viewed as an element in $\Latin(m_1 + n_1, \dots, m_d + n_d)$.
    The tensor product of bato tensors is bato: $\cS \otimes \cT \in X_{L \times \tilde L}$ with $L \times \tilde L \in \Latin(m_1, \dots, m_d, n_1, \dots, n_d)$.
    The Kronecker product of bato tensors is bato: $\cS \boxtimes \cT \in X_{L \boxtimes \tilde L}$ where $L \boxtimes \tilde L = \{((i_1, j_1), \dots, (i_d, j_d)) \mid \bi \in L, \bj \in \tilde L\}$ can be viewed as an element in $\Latin(m_1n_1, \dots, m_dn_d)$ by identifying $[m_k] \times [n_k]$ with $[m_kn_k]$.
    
    Since the pair $(\cS, \cT)$ is generic, we deduce that $\cS \oplus \cT$, $\cS \otimes \cT$ and $\cS \boxtimes \cT$ satisfy the genericity conditions from \Cref{thm:uniqueness-bato} given in \eqref{eq:genericity-uniqueness}, so they have a unique bato decomposition. Hence,
    \begin{align*}
         & \batorank(\cS \oplus \cT) = |L \sqcup \tilde L| = |L| + |\tilde L| = \batorank(\cS) + \batorank(\cT), \\
         &  \batorank(\cS \otimes \cT) = |L \times \tilde L| = |L| \cdot |\tilde L| = \batorank(\cS) \cdot \batorank(\cT),\\
         & \batorank(\cS \boxtimes \cT) = |L \boxtimes \tilde L| = |L| \cdot |\tilde L| = \batorank(\cS) \cdot \batorank(\cT). \qedhere
    \end{align*}
\end{proof}

In light of \Cref{prop:bato-rank-additivie-multiplicative} and the fact that these identities always hold for matrices, we suspect that the bato rank always behaves additively and multiplicatively under taking direct sums and tensor products, respectively.

\begin{conjecture}
    $\batorank(\cS \oplus \cT) = \batorank(\cS) + \batorank(\cT)$ for all bato $\cS$ and $\cT$ of order $d \geq 2$.
\end{conjecture}

\begin{conjecture} \label{conj:kronecker-product-bato-rank}
    $\batorank(\cS \otimes \cT) = \batorank(\cS) \cdot \batorank(\cT)$ for all bato $\cS$ and $\cT$.
\end{conjecture}

However, bato rank is not always multiplicative under taking Kronecker products.

\begin{example} \label{ex:kronecker-submultiplicative}
    Let $\cT = \be_1 \otimes \be_1 \otimes \be_1 + \be_1 \otimes \be_2 \otimes \be_2 + \be_2 \otimes \be_1 \otimes \be_2 - \be_2 \otimes \be_2 \otimes \be_1 \in (\R^2)^{\otimes 3}$, which has bato rank 4 (e.g., by \Cref{prop:bato-rank-slices}). Define
    \begin{align*}
        & \bu_1 = \frac{1}{\sqrt{2}}(\be_1 \otimes \be_1 - \be_2 \otimes \be_2), \quad  \bu_2 = \frac{1}{\sqrt{2}}(\be_1 \otimes \be_2 + \be_2 \otimes \be_1), \\
        & \bu_3 = \frac{1}{\sqrt{2}}(\be_1 \otimes \be_1 + \be_2 \otimes \be_2), \quad \bu_4 = \frac{1}{\sqrt{2}}(\be_1 \otimes \be_2 - \be_2 \otimes \be_1).
    \end{align*}
    Then, $\bu_1, \bu_2, \bu_3, \bu_4 \in \R^2 \otimes \R^2$ are orthonormal and
    \begin{align*}
      \cT \boxtimes \cT =   \sqrt2\bigl(&
\bu_1\otimes \bu_1\otimes \bu_1
+\bu_1\otimes \bu_2\otimes \bu_2
+\bu_2\otimes \bu_1\otimes \bu_2
-\bu_2\otimes \bu_2\otimes \bu_1\\
+
&\bu_3\otimes \bu_3\otimes \bu_3
+\bu_3\otimes \bu_4\otimes \bu_4
+\bu_4\otimes \bu_3\otimes \bu_4
-\bu_4\otimes \bu_4\otimes \bu_3
\bigr).
    \end{align*}
    Hence, $\batorank(\cT \boxtimes \cT) \leq 8 < 16 = \batorank(\cT)^2$.
\end{example}

\begin{remark}
    The tensor $\cT$ in \Cref{ex:kronecker-submultiplicative} is the structure tensor of the real algebra of complex numbers $\C$, and $\cT \boxtimes \cT$ is the structure tensor of the bicomplex numbers $\C \otimes_\R \C$; see \Cref{sec:world} for more details.
\end{remark}

We conclude this subsection by relating the bato rank with tensor norms and the ranks of slices. These results will be useful later in \Cref{sec:world}. Recall that the Frobenius norm of
 a tensor $\cT \in \R^{n_1 \times \cdots \times n_d}$ is $\| \cT\| = \sqrt{\sum_{\bi} t_{\bi}^2}$, and its spectral norm is
\[
\| \cT \|_\sigma = \max_{(\bu^{(1)}, \dots, \bu^{(d)}) \in \bbS^{n_1-1} \times \cdots \times \bbS^{n_d-1}} |\cT(\bu^{(1)}, \dots, \bu^{(d)})|,
\]
which coincides with the largest singular value of $\cT$ in absolute value \cite{lim2005singular}.
\begin{proposition} \label{prop:bato-rank-lower-bound}
    For every bato tensor $\cT \neq 0$, 
    \[
    \batorank(\cT) \geq \frac{\| \cT\|^2}{ \| \cT\|_\sigma^2}.
    \]
\end{proposition}
\begin{proof}
   Consider a bato decomposition $\cT = \sum_{\bi \in L} s_{\bi} \bu_{i_1}^{(1)} \otimes \cdots \otimes \bu_{i_d}^{(d)}$ as in \Cref{cor:bato-support}. 
   Two-orthogonality implies that $\| \cT \|^2 = \sum_{\bi \in L} s_{\bi}^2$ and $|\cT(\bu_{i_1}^{(1)}, \dots, \bu_{i_d}^{(d)})| = |s_{\bi}|$ for all $\bi \in L$. Therefore, $|s_{\bi}| \leq \| \cT\|_\sigma$ for all $\bi \in L$, so $\| \cT\|^2 = \sum_{\bi \in L}s_\bi^2 \leq |L| \| \cT\|_\sigma^2$.
\end{proof}

\begin{proposition}[Matrix-rank formula for bato slices]
\label{prop:bato-rank-slices}
Let $\T\in \mathbb R^{n_1}\otimes \mathbb R^{n_2}\otimes \mathbb R^{n_3}$
have a bato decomposition
$\T=\sum_{(i,j,k)\in L} s_{ijk} \bu_i\otimes\bv_j\otimes\bw_k,$
where \(L\in\Latin(n_1,n_2,n_3)\), all \(s_{ijk}\neq0\), and
\(\{\bu_i\}\), \(\{\bv_j\}\), and \(\{\bw_k\}\) are orthonormal bases. Then
\[
\begin{aligned}
    |L|=
    \sum_{i=1}^{n_1}
    \operatorname{rank}
    \T(\bu_i,\mathord\cdot,\mathord\cdot)=
    \sum_{j=1}^{n_2}
    \operatorname{rank}
    \T(\mathord\cdot,\bv_j,\mathord\cdot)=
    \sum_{k=1}^{n_3}
    \operatorname{rank}
    \T(\mathord\cdot,\mathord\cdot,\bw_k),
\end{aligned}
\]
where each contraction is regarded as a matrix in the remaining two factors.
\end{proposition}

\begin{proof}
Fix \(i\in[n_1]\). Contracting in the first factor gives
\[
    \T(\bu_i,\mathord\cdot,\mathord\cdot)
    =
    \sum_{\substack{(i,j,k)\in L}}
    s_{ijk}\,\bv_j\otimes\bw_k.
\]
If \((i,j,k),(i,j',k')\in L\) are distinct, then the Latin condition
implies \(j\neq j'\) and \(k\neq k'\). Hence, the matrix $\T(\bu_i,\mathord\cdot,\mathord\cdot)$ in
the bases \(\{\bv_j\}\) and \(\{\bw_k\}\) has at most one nonzero entry in each
row and column, so its rank is $\left|\{(i,j,k)\in L\}\right|$.
Summing over \(i\) gives the desired result.
The other two identities follow by the same argument.
\end{proof}

\subsection{Truncations of bato decompositions} \label{sec:truncations}

This section is motivated by the following question. Given a rank-$r$ tensor $\cT = \sum_{i=1}^r \bu^{(1)}_i \otimes \cdots \otimes \bu_i^{(d)} \in \R^{n_1} \otimes \cdots \otimes \R^{n_d}$, when is the truncation $\tilde \cT = \sum_{i=1}^k \bu^{(1)}_i \otimes \cdots \otimes \bu_i^{(d)}$ a best rank-$k$ approximation of $\cT$?

Two-orthogonality is necessary but insufficient for this to occur \cite{vannieuwenhoven2014generic}. If $\cT$ is not two-orthogonal, best rank-$k$ and rank-$\tilde k$ approximations of $\cT$ need not share any rank-one terms in their minimal CP decompositions. This poses a problem in applications when one computes a low-rank approximation of a data tensor and analyzes the rank-one terms obtained.

We argue that truncations of bato decompositions have desirable properties, as opposed to truncations of general CP decompositions.
Given $L \in \Latin(n_1, \dots, n_d)$ and a bato decomposition $\cT = \sum_{\bi \in L} s_\bi \bu_{i_1}^{(1)} \otimes \cdots \otimes \bu_{i_d}^{(d)}$, a truncation of this decomposition is of the form $\tilde \cT = \sum_{\bi \in {\tilde L}} s_\bi \bu_{i_1}^{(1)} \otimes \cdots \otimes \bu_{i_d}^{(d)}$ for $\tilde{L} \subset L$. Truncations of bato decompositions are bato.
If the bato rank of $\cT$ is $|L|$, then the truncation $\tilde \cT$ gives an approximation of $\cT$ of lower bato rank. Given a budget $k < |L|$, it is natural to truncate $\cT$ taking $\tilde L \subset L$ with $|\tilde L| = k$ such that $|s_{\bi}| \geq |s_{\bj}|$ for all $\bi \in \tilde{L}$ and all $\bj \in L \setminus \tilde L$, so that $\| \cT - \tilde \cT\|^2 = \sum_{\bi \in L \setminus \tilde L} s_{\bi}^2$.

A \emph{best $\tilde{L}$-bato approximation} of $\cT$ is a solution to the optimization problem
\begin{equation} \label{eq:optimization-problem}
    \min_{\tilde \cT \in X_{\tilde L}} \| \cT - \tilde \cT \|^2,
\end{equation}
which always exists because $X_{\tilde L}$ is Euclidean closed (\Cref{prop:closed-semialgebraic}). A \emph{critical $\tilde L$-bato approximation} of $\cT$ is a critical point of \eqref{eq:optimization-problem}, i.e., a smooth point $\tilde \cT \in X_{\tilde L}$ such that $\cT - \tilde \cT$ is orthogonal to the tangent space of $X_{\tilde L}$ at $\tilde \cT$ (see \cite{draisma2016euclidean} for details). We show that truncating a generic $L$-bato tensor with $\tilde L \subset L$ gives a critical $\tilde L$-bato approximation.

\begin{example}[Matrix low-rank approximation]
     Let $L = \{ (i, i) \mid i \in [r]\} \in \Latin(m, n)$ and consider the SVD of a matrix $\cT = \sum_{i \in [r]} s_{ii} \bu_i \otimes \bv_i \in \R^{m \times n}$ with $s_{11} > \cdots > s_{rr} > 0$. Take $\tilde L = \{ (i, i) \mid i \in [\tilde r]\}$ for $\tilde r < r$. Then, the solution to \eqref{eq:optimization-problem} is given by the truncation $\tilde \cT = \sum_{i \in [\tilde r]} s_{i,i} \bu_i \otimes \bv_i$, and the other critical points of \eqref{eq:optimization-problem} are given by all truncations of $\cT$ with $\tilde r$ summands \cite{eckart1936approximation}.
\end{example}

For a positive integer $n$, let
$\mathfrak{so}(n) = \{ \Omega \in \R^{n \times n} \mid \Omega + \Omega^\top = 0\}$
be the Lie algebra of $\rSO(n)$, that is, the tangent space of $\rSO(n)$ at the identity $I_n$. Given $\cT \in \R^{n_1 \times \cdots \times n_d}$ and $k \in [d]$, we use the notation
\[
\mathfrak{so}(n_k)\cdot \cT := \left\{(I_{n_1}, \dots, I_{n_{k-1}}, \Omega, I_{n_{k+1}}, \dots, I_{n_d}) \cdot \cT \mid \Omega \in \mathfrak{so}(n_k) \right\}. 
\]

\begin{lemma}[Tangent space to bato variety]
\label{lem:tangent-fixed-support-bato}
Let $L \in \Latin(n_1, \dots, n_d)$ 
and consider a generic $L$-bato tensor  
$\cT = \sum_{\bi \in L} s_{\bi}\,
    \bu^{(1)}_{i_1}\otimes\cdots\otimes \bu^{(d)}_{i_d} \in X_L \subset \R^{n_1 \times \cdots \times n_d}$.
Then, $\cT$ is a smooth point of $X_L$ and the tangent space of~$X_L$ at $\cT$ is
    \[
    T_{\cT}X_L = \operatorname{span}
    \left\{
    \bu^{(1)}_{i_1}\otimes\cdots\otimes\bu^{(d)}_{i_d}
    \mid 
    (i_1,\ldots,i_d)\in L
    \right\} + \sum_{k=1}^d \mathfrak{so}(n_k)\cdot \cT.
    \]
\end{lemma}

\begin{proof}
Let $V_L(\bu) = \operatorname{span}
    \{
    \bu^{(1)}_{i_1}\otimes\cdots\otimes\bu^{(d)}_{i_d}
    \mid 
    (i_1,\ldots,i_d)\in L
    \}$ and consider the parametrization of $L$-bato tensors
    \[
    \begin{array}{cccc}
         \Phi : & \rSO(n_1) \times \cdots \times \rSO(n_d) \times V_L(\bu) & \to & X_L \subset \R^{n_1 \times \cdots \times n_d}, \\
         & (Q^{(1)}, \dots, Q^{(d)}, \cS) &  \mapsto & (Q^{(1)}, \dots, Q^{(d)}) \cdot \cS.
    \end{array}
    \]
    A tangent vector to \(\rSO(n_k)\) at the identity is an
element \( \Omega^{(k)} \in \mathfrak{so}(n_k)\), and a tangent vector to \(V_L(\bu)\) at $\cT$ is an
arbitrary element \( \dot\cS \in V_L(\bu)\), so the differential of $\Phi$ at
\((I_{n_1},\ldots,I_{n_d},\cT)\)~is
\[
d\Phi_{(I_{n_1},\ldots,I_{n_d},\cT)} (\Omega^{(1)}, \dots, \Omega^{(d)}, \dot \cS) = \dot \cS + \sum_{k=1}^d (I_{n_1}, \dots, I_{n_{k-1}}, \Omega^{(k)}, I_{n_{k+1}}, \dots, I_{n_d}) \cdot \cT.
\]
Hence, $\cT$ is a smooth point of $X_L$ for being generic, and the tangent space of $X_L$ at $\cT$ is
\[
    T_{\cT} X_L = \operatorname{im} d\Phi_{(I_{n_1},\ldots,I_{n_d},\cT)}
    =
    V_L(\bu)
    +
    \sum_{k=1}^d\mathfrak{so}(n_k)\cdot\cT. \qedhere
\]
\end{proof}

\begin{proof}[Proof of \Cref{thm:truncation-bato-critical}]
    Fix $L \in \Latin(n_1, \dots, n_d)$ and consider a generic $L$-bato decomposition $\cT = \sum_{\bi \in L} s_{\bi} \bu_{i_1}^{(1)} \otimes \cdots \otimes \bu_{i_d}^{(d)} \in X_L \subset \R^{n_1 \times \cdots \times n_d}$. We want to show that, for any $\tilde L \subsetneq L$, the truncation $\tilde \cT = \sum_{\bi \in \tilde L} s_{\bi} \bu_{i_1}^{(1)} \otimes \cdots \otimes \bu_{i_d}^{(d)}$ is a critical $\tilde L$-bato approximation of $\cT$.

    Write $\cU_{\bi} = \bu^{(1)}_{i_1}\otimes\cdots\otimes \bu^{(d)}_{i_d}$ for each $\bi \in L$. Then, $\cT - \tilde \cT =\sum_{\bi\in L\setminus \tilde L}s_{\bi} \cU_{\bi}$ is orthogonal to the tangent space of $X_{\tilde L}$ at $\tilde \cT$ given in \Cref{lem:tangent-fixed-support-bato} by two-orthogonality. Indeed, every rank-one summand $s_{\bi}\,  \cU_{\bi}$ for $\bi \in L \setminus \tilde L$ is orthogonal to the linear spaces $\operatorname{span}\{\cU_{\bi} \mid \bi \in \tilde L\}$ and~$\mathfrak{so}(n_k) \cdot \tilde \cT$ for all $k \in [d]$:
    \[
    \langle \cU_{\bi}, (I_{n_1}, \dots, I_{n_{k-1}}, \Omega, I_{n_{k+1}}, \dots, I_{n_d}) \cdot \cU_{\bj}\rangle = \langle \bu_{i_k}^{(k)}, \Omega \bu_{j_k}^{(k)} \rangle \prod_{\ell \neq k} \langle \bu_{i_\ell}^{(\ell)}, \bu_{j_\ell}^{(\ell)} \rangle = 0 \qquad \text{if } \bi \neq \bj,
    \]
    so $\cT - \tilde \cT$ is orthogonal to the sum of these linear spaces.
    
    Let $r = |L|$ and $\tilde r = |\tilde L|$. By genericity of $\cT$ and \Cref{thm:uniqueness-bato}, $\tilde \cT$ has a unique bato decomposition and is a smooth point in the set of tensors of bato rank at most $\tilde r$, so it is a critical bato-rank-$\tilde r$ approximation of $\cT$.
\end{proof}

Unfortunately, bato truncations are not always optimal.

\begin{example}
    Take $\cT = \lambda e_1^{\otimes 3} + \mu(e_1 \otimes e_2 \otimes e_2 + e_2 \otimes e_1 \otimes e_2 + e_2 \otimes e_2 \otimes e_1) \in (\R^{2})^{\otimes 3}$ with~$0 < \lambda < \mu$, which has bato-rank four. 
    Consider the rank-one tensor $\cS = \frac{\mu}{2}(e_1 + e_2)^{\otimes 3}$. Then
    \(
    \| \cT - \cS\|^2 = \lambda^2 - \lambda \mu + 2\mu^2 < \lambda^2 + 2\mu^2 < 3\mu^2,
    \)
    so a best rank-one approximation of $\cT$ is not given by any of the rank-one terms in its bato decomposition.
\end{example}

Sufficient conditions for truncations of two-orthogonal decompositions to give optimal low-rank approximations 
are given by \cite{vannieuwenhoven2014generic}, but these conditions are not necessary \cite{ribot2026decomposing}.
Our setup is different because we consider low-bato-rank approximations instead of low-rank approximations, but those sufficient conditions also apply to our case. Characterizing which bato decompositions lead to optimal truncations (with respect to bato tensors) is an interesting direction for future work.

\begin{question}
    Let $L \in \Latin(n_1, \dots, n_d)$ with $|L| = r$ and let $\cT = \sum_{\bi \in L} s_{\bi} \be_{i_1} \otimes \cdots \otimes \be_{i_d}$. Consider a chain of partial Latin hyperrectangles
    \[
    L_1 \subset L_2 \subset \cdots \subset L_r = L
    \]
    with $|L_k| = k$.
    For which values $\{s_{\bi} \mid \bi \in L \}$ is the truncation $\tilde \cT_k = \sum_{\bi \in L_k} s_{\bi} \be_{i_1} \otimes \cdots \otimes \be_{i_d}$ a best $L_k$-bato approximation of $\cT$ for each $k \in [r]$?
\end{question}

\section{The bato variety} \label{sec:bato-variety}

In this section, we study the set of bato tensors in $\R^{n_1} \otimes \cdots \otimes \R^{n_d}$ and the algebraic variety they define. Given $ L \in \Latin(n_1, \dots, n_d)$, consider the linear space
	\[
	V_L = \{  \cS \in \R^{n_1} \otimes \cdots \otimes \R^{n_d} \mid s_{\bi} = 0 \text{ for all } \bi \not \in L\},
	\]
    so that the set of $L$-bato tensors is its orbit under the action of the group $\rO(n_1) \times \cdots \times \rO(n_d)$:
\[
X_L = \left(\rO(n_1) \times \cdots \times \rO(n_d) \right) \cdot V_L.
\]
The \emph{bato variety} is
\[
\overline{X} = \bigcup_{L \in \Latin(n_1, \dots, n_d)} \overline{X_L}.
\]
Throughout the paper, overlines denote closure with respect to the Zariski topology.

\begin{example}
    The bato variety $\overline{X} \subset (\R^{2})^{\otimes 3}$ has two irreducible components: one of dimension $5$ and one of dimension $7$. They are given by the maximal partial Latin squares
\begin{center}
$L_1 = $
\begin{tabular}{|c|c|}
\hline
1 &  \\
\hline
 & 2 \\
\hline
\end{tabular} \ ,
\quad
$L_2 = $
\begin{tabular}{|c|c|}
\hline
1 & 2 \\
\hline
2 & 1 \\
\hline
\end{tabular} \ .
\end{center}
The component corresponding to $L_1$ is the odeco variety, which is defined by
\begin{align*}
    &t_{121}t_{211}+t_{122}t_{212}-t_{111}t_{221}-t_{112}t_{222} = 0 ,\\
    &t_{112}t_{211}-t_{111}t_{212}+t_{122}t_{221}-t_{121}t_{222} = 0,\\
    &t_{112}t_{121}-t_{111}t_{122}+t_{212}t_{221}-t_{211}t_{222} = 0.
\end{align*}
The component corresponding to $L_2$ is defined by
\[
\sum_{(i,j,k) \in \{1,2\}^3} (-1)^{i+j+k} \, 
    t_{i,j,k} \, g(t_{i+1,j,k}, t_{i, j+1, k}, t_{i,j,k+1}, t_{i+1, j+1, k+1}) = 0,
\]
where the sum in the indices is taken modulo $2$ and $g$ defines the elliptope:
\begin{equation*} \label{eq:elliptope}
    g(w,x,y,z) = 2\,wxy + w^2z + x^2z + y^2z - z^3.
\end{equation*}
\end{example}

\begin{proposition} \label{prop:closed-semialgebraic}
For all $L \in \Latin(n_1, \dots, n_d)$
the set $X_L \subset{\R^{n_1 \times \cdots \times n_d}}$ is semialgebraic and closed with respect to the Euclidean topology.
\end{proposition}

\begin{proof}
	The squared distance from a tensor $\cT \in \R^{n_1} \otimes \cdots \otimes \R^{n_d}$ to the linear space $V_L$ is
	\[
	\operatorname{dist}_{V_L}(\cT)^2 = \sum_{\bi \notin L}\cT_\bi^2,
	\]
	so $\cT \in X_L$ if and only if 
	\[
	\min_{(U^{(1)}, \dots, U^{(d)}) \in \rO(n_1) \times \cdots \times \rO(n_d)} \operatorname{dist}_{V_L}((U^{(1)}, \dots, U^{(d)}) \cdot \cT)^2 = 0.
	\]
	Hence, $X_L$ is Euclidean closed by compactness of $\rO(n_1) \times \cdots \times \rO(n_d)$ and continuity of $\operatorname{dist}_{V_L}$ and the group action. The set $X_L$ is semialgebraic by the Tarski-Seidenberg theorem~\cite{seidenberg}.
\end{proof}
\begin{corollary} \label{cor:bato-closed}
    The set of tensors of bato rank at most $r$ is semialgebraic and Euclidean closed for all $r \geq 1$.
\end{corollary}
\begin{proof}
	Given $r \geq 1$, the set of bato tensors of bato rank at most $r$ is the union of $X_L$ over all $L \in \Latin(n_1, \dots, n_d)$ such that $|L| \leq r$, so the statement follows from \Cref{prop:closed-semialgebraic} because there are finitely many partial Latin hyperrectangles.
\end{proof}

We suspect that the following is also true.
\begin{conjecture}\label{conj:bato-zariski-closed}
	The set $X_L$ is Zariski closed for all $L \in \Latin(n_1, \dots, n_d)$.
\end{conjecture}
Assuming \Cref{conj:bato-zariski-closed} holds, the set of bato tensors is Zariski closed. \Cref{conj:bato-zariski-closed} holds for matrices $(d=2)$, since in that case $X_L = \{M \mid \rank (M) \leq |L|\}$, which is a determinantal variety.
More evidence for \Cref{conj:bato-zariski-closed} is given in \cite{boralevi2017orthogonal}, where the authors show that the set of odeco tensors forms a real algebraic variety. That is, they show that~$X_L$ is Zariski closed for $L = \{(i, \dots, i) \mid i \leq \min_k\{n_k\}\}$. They also show that the set of alternatingly odeco tensors is Zariski closed, and alternatingly odeco tensors are bato.

\subsection{Dimension} Here we compute the dimension of the bato variety $\overline{X} \subset \R^{n_1} \otimes \cdots \otimes \R^{n_d}$. For each $k \in [d]$, let
\[
\pi_k : [n_1] \times \cdots \times [n_d] \to [n_k], \; (i_1, \dots, i_d) \mapsto i_k
\]
be the projection onto the $k$-th factor. 

\begin{lemma}
\label{lem:dimension-fixed-support}
Let \(L\in \Latin(n_1, \dots,  n_d) \) and let \(m_k=|\pi_k(L)|\) for each $k \in [d]$. Then,
\[
    \dim X_L
    =
    |L|+\sum_{k=1}^d
    \left(n_km_k-\binom{m_k+1}{2}\right).
\]
\end{lemma}

\begin{proof}
Points in $X_L$ are of the form $\cT = \sum_{\bi \in L} s_{\bi} \bu_{i_1} \otimes \cdots \otimes \bu_{i_d}$.
For each factor $k \in [d]$, a set of $m_k$ orthonormal vectors is a point in the Stiefel manifold
\[
\operatorname{St}(m_k,n_k)
    =
    \{ U \in\mathbb R^{n_k\times m_k}: U^\top U=I_{m_k}\},
\]
    which has dimension $n_km_k-\binom{m_k+1}{2}$.
The coefficients \(s_{\bi}\), \(\bi\in L\), contribute \(|L|\) parameters.
By \Cref{thm:uniqueness-bato} and the fiber dimension theorem, the dimension of $X_L$ is the dimension of the
parameter space, namely
\[
    |L|+\sum_{k=1}^d
    \left(n_km_k-\binom{m_k+1}{2}\right). \qedhere
\]
\end{proof}

\begin{proof}[Proof of \Cref{thm:dimension-bato}]
The bato variety is the union of all $\overline{X_L}$ for $L \in \Latin(n_1, \dots, n_d)$. 
By \Cref{lem:dimension-fixed-support}, we obtain components $\overline{X_L}$ of maximal dimensions by picking $L$ such that $|L| = \prod_{k=1}^{d-1} n_k$, $m_k = n_k$ for all $k \leq d-1$, and $m_d = \min\{\prod_{k=1}^{d-1} n_k, n_d\}$.
If $n_d  \geq \prod_{k=1}^{d-1} n_k$, we can take any  injection $f : [n_1] \times \cdots \times [n_{d-1}] \to [m_d]$ and let
\[
L = \{(i_1, \dots, i_{d-1}, f(i_1, \dots, i_{d-1})) \mid i_k \in [n_k]\}.
\]
If $n_d < \prod_{k=1}^{d-1} n_k$, we can start with the $L$ given in the proof of \Cref{prop:maximal-bato-rank} and change some tuples $(i_1, \dots, i_{d-1}, i_d)$ by $(i_1, \dots, i_{d-1}, i_d')$ if needed to ensure that $\pi_d(L) = [n_d]$.
By counting parameters and using \Cref{lem:dimension-fixed-support} we obtain the desired result:
\[
\dim \overline X = \prod_{k=1}^{d-1} n_k + \sum_{k=1}^{d-1} \binom{n_k}{2} + n_dm_d - \binom{m_d+1}{2}. \qedhere
\]
\end{proof}

\begin{corollary}
    A generic tensor in $\R^{n_1 \times \cdots \times n_d}$ is not bato when $d \geq 3$ and $n_1, \dots, n_d\geq 2$.
\end{corollary}

\begin{remark}
    In the cubical case, when $n_1 = \cdots = n_d$, we obtain that the dimension of the bato variety is $n^{d-1} + d \binom{n}{2}$. This coincides with the lower bound of the dimension of the two-orthogonal variety given in \cite[Theorem 1.4]{ribot2026decomposing}. Indeed, their proof parametrizes a set of tensors of maximal bato rank and shows that the Jacobian of the parametrization is full rank.
    The finite-fiber argument used here avoids a direct Jacobian computation and extends readily to rectangular formats.
\end{remark}

\subsection{Irreducible components}

For a positive integer $n$, let $\mathfrak S_n$ denote the symmetric group on $[n]$. There is a natural group action of $\fS_{n_1} \times \cdots \times \fS_{n_d}$ on $\Latin(n_1, \dots, n_d)$ given by 
\[
(\sigma_1, \dots, \sigma_d) \cdot L = \{(\sigma_1(i_1), \dots, \sigma_d(i_d)) \mid (i_1, \dots, i_d) \in L\}.
\]
The orbits of this group action are called the \emph{isotopy classes}, i.e., $L, \tilde{L} \in \Latin(n_1, \dots, n_d)$ are \emph{isotopic} if $L = (\sigma_1, \dots, \sigma_d) \cdot \tilde{L}$ for some $(\sigma_1, \dots, \sigma_d) \in \fS_{n_1} \times \cdots \times \fS_{n_d}$.

\begin{proposition} \label{prop:latin-irreducible}
    The algebraic variety $\overline{X_L}$ is irreducible for each $L \in \Latin(n_1, \dots, n_d)$.
\end{proposition}
\begin{proof}
    For a positive integer $n$, let $\rSO(n) = \{U \in \rO(n) \mid \det(U) = 1\}$ denote the special orthogonal group.
    We have the equality of orbits
    \[
    (\rO(n_1) \times \cdots \times \rO(n_d) ) \cdot V_L = (\rSO(n_1) \times \cdots \times \rSO(n_d)) \cdot V_L,
    \]
    because given $\cT = (U^{(1)}, \dots, U^{(d)}) \cdot \cS$ we may flip the sign of the first column of $U^{(k)}$ and the first $k$-slice of $\cS$ 
    leaving the value $(U^{(1)}, \dots, U^{(d)}) \cdot \cS$ unchanged. So $\overline{X_L}$ is the closure of the image of an irreducible variety under a polynomial map, hence irreducible.
\end{proof}
\begin{remark}
Given $L \in \Latin(n_1, \dots, n_d)$, the set $X_L$ is connected with respect to the Euclidean topology for being the image of a connected set under a continuous map.
\end{remark}

In order to characterize the irreducible components of the bato variety, it is beneficial to work over the complex numbers, where the Zariski and Euclidean closures of our sets of interest coincide.
Given $L \in \Latin(n_1, \dots, n_d)$, we denote the complexification of the set~$X_L \subset \R^{n_1 \times \cdots \times n_d}$ by 
\[
X_L^\C = (\rO(n_1, \C) \times \cdots \times \rO(n_d, \C)) \cdot V_L^\C \subset \C^{n_1 \times \cdots \times n_d},
\]
where
$\rO(n, \C) = \{ U \in \C^{n \times n} \mid U^\top U = I\}$ and $V_L^\C = \{ \cS \in \C^{n_1 \times \cdots \times n_d} \mid s_\bi = 0 \text{ for all } \bi \notin L\}$. Crucially, the Zariski and Euclidean closures of $X_L^\C$ coincide, by Chevalley's theorem.

\begin{theorem}[Isotopy invariants] \label{thm:isotopy-invariants}
    Two partial Latin hyperrectangles $L$ and $\tilde L$ are isotopic if and only if $\overline{X_L} = \overline{X_{\tilde L}}$.
\end{theorem}
\begin{proof}
    Suppose that $L$ and $\tilde L$ are isotopic. Then, there exists a tuple of permutation matrices $(U^{(1)}, \dots, U^{(d)}) \in \rO(n_1) \times \cdots \times \rO(n_d)$ such that $V_L = (U^{(1)}, \dots, U^{(d)}) \cdot V_{\tilde L}$, so $\overline{X_{L}} = \overline{X_{\tilde L}}$.

    Conversely, assume $\overline{X_L} = \overline{X_{\tilde L}}$. This is equivalent to $\overline{X_{L}^\C} = \overline{X_{\tilde L}^\C}$, because $\overline{X_L^\C}$ and $\overline{X_{\tilde L}^\C}$ are irreducible and contain real smooth points.
    Therefore, $V_L^\C \subseteq \overline{X_{\tilde L}^\C}$.
    
    Take $\cT=\sum_{\bi\in L}s_{\bi}e_{i_1}^{(1)} \otimes \cdots \otimes e_{i_d}^{(d)}\in V_L^\C$ such that \(s_{\bi}\neq0\), and such that for every $k \in [d]$ and every $a \in \pi_k(L)$, the numbers  $\lambda^{(k)}_{a}=\sum_{\bi\in L \mid i_k=a}s_{\bi}^2,$ are nonzero and pairwise distinct.
    For $a \notin \pi_k(L)$, we define $\lambda^{(k)}_a$ to be zero.
    Since $\cT \in \overline{X_{\tilde L}^\C}$, there exists a sequence of tensors $\cT_m= (U_m^{(1)}, \dots, U_m^{(d)})\cdot \tilde \cS_m$ with $U_m^{(k)} \in \rO(n_k, \C)$ and~$\tilde \cS_m = \sum_{\bi \in \tilde L} \tilde{s}_{\bi, m} \be_{i_1} \otimes \cdots \otimes \be_{i_d} \in V_{\tilde L}^\C$ such that
    $\cT_m\xrightarrow{m \to \infty} \cT$ in the Euclidean topology.
    
    For each $k \in [d]$, the mode-$k$ Gram matrix of the tensor $\cT_m$ is
    \[
    T_m^{(k)} {T_m^{(k)}}^\top = U^{(k)}_m \tilde \Lambda^{(k)}_m {U_m^{(k)}}^\top
    \]
    where $\tilde \Lambda_m^{(k)} = \diag (\tilde \lambda^{(k)}_{1, m}, \dots, \tilde \lambda^{(k)}_{n_k, m} )$ with $\tilde \lambda^{(k)}_{a, m} = \sum_{\bi \in \tilde L \mid i_k = a} \tilde s_{\bi, m}^2$. Similarly, the mode-$k$ Gram matrix of $\cT$ is
    \[T^{(k)} {T^{(k)}}^\top = \diag(\lambda^{(k)}_1, \dots, \lambda^{(k)}_{n_k}),
    \]
    and we obtain $T_m^{(k)} {T_m^{(k)}}^\top \xrightarrow{m \to \infty} T^{(k)} {T^{(k)}}^\top$ by continuity.
    After passing to a subsequence if needed, for each $k \in [d]$ there exist injections $\alpha_k : \pi_k(L) \to [n_k]$ such that $\tilde \lambda^{(k)}_{\alpha_k(a), m} \xrightarrow{m \to \infty} \lambda_{a}^{(k)}$ for all $a \in \pi_k(L)$, because the eigenvalues $\{ \lambda_a^{(k)} \mid a \in \pi_k(L)\}$ are distinct. These injections may be extended to permutation of $[n_k]$.

    Fix $\bi \in L$ and, for each $k \in [d]$, consider the sequence of orthogonal projections $P^{(k)}_{i_k, m}$ from $\C^{n_k}$ onto the lines spanned by $\bu_{\alpha_k(i_k), m}^{(k)}$:
    the sequence of eigenvectors of $T_m^{(k)} {T_m^{(k)}}^\top$ whose eigenvalues tend to $\lambda^{(k)}_{i_k}$. We have $P^{(k)}_{i_k, m} \xrightarrow{m \to \infty} E_{i_k}$, where $E_{i_k}$ is the coordinate projection onto the line spanned by $\be_{i_k}$. By continuity we get
    \[
    \begin{array}{ccc}
      \underbrace{\tilde{s}_{\alpha_1(i_1), \dots, \alpha_d(i_d), m} \bu^{(1)}_{\alpha_1(i_1), m} \otimes \cdots \otimes \bu^{(d)}_{\alpha_d(i_d), m}}_{\, \rotatebox{90}{=}}  & \xrightarrow{m \to \infty} & \underbrace{s_{\bi} \be_{i_1} \otimes \cdots \otimes \be_{i_d}}_{\, \rotatebox{90}{=}}. \\
    (P^{(1)}_{i_1, m}, \dots, P^{(d)}_{i_d, m}) \cdot \cT_m & & (E_{i_1}, \dots, E_{i_d}) \cdot \cT
    \end{array}
    \]
    Hence, $(\alpha_1(i_1), \dots, \alpha_d(i_d)) \in \tilde{L}$ for all $\bi \in L$, so $L$ is isotopic to a subset of $\tilde L$. By symmetry, since $\overline{X_L} = \overline{X_{\tilde L}}$, we deduce that $\tilde{L}$ is isotopic to a subset of $L$, so $L$ and $\tilde{L}$ are isotopic.
\end{proof}

The proof above implies the following.

\begin{corollary} \label{cor:isotopic-subset}
    One has $\overline{X_{\tilde L}} \subseteq \overline{X_L}$ if and only if $\tilde{L}$ is isotopic to a subset of $L$.
\end{corollary}

As a consequence, we deduce that the irreducible components of the bato variety are indexed by isotopy classes of maximal partial Latin hyperrectangles.

\begin{proof}[Proof of \Cref{thm:irreducible-components-bato}]
For every $L\in\Latin(n_1,\ldots,n_d)$, the variety
$\overline{X_L}$ is irreducible by \Cref{prop:latin-irreducible}. If $L \subsetneq \tilde L$, then $\overline{X_L} \subsetneq \overline{X_{\tilde L}}$, so non-maximal partial Latin hyperrectangles do not give irreducible components of the bato variety. If $L$ and $\tilde L$ are maximal, then $\overline{X_L} = \overline{X_{\tilde L}}$ if and only if $L$ and $\tilde L$ are isotopic, by \Cref{thm:isotopy-invariants}.
\end{proof}

\begin{example}
There are four maximal isotopy classes in $\Latin(3,3,3)$ with representatives:
\begin{center}
\begin{tabular}{|c|c|c|}
\hline
1 &  &  \\
\hline
 & 2 & 3 \\
\hline
 & 3 & 2 \\
\hline
\end{tabular}
\quad
\begin{tabular}{|c|c|c|}
\hline
1 & 2 &  \\
\hline
2 & 1 & 3 \\
\hline
 & 3 & 1 \\
\hline
\end{tabular}
\quad
\begin{tabular}{|c|c|c|}
\hline
1 & 2 &  \\
\hline
2 & 1 & 3 \\
\hline
 & 3 & 2 \\
\hline
\end{tabular}
\quad
\begin{tabular}{|c|c|c|}
\hline
1 & 2 & 3 \\
\hline
2 & 3 & 1 \\
\hline
3 & 1 & 2 \\
\hline
\end{tabular}
\end{center}
Therefore, the bato variety $\overline{X} \subset (\R^3)^{\otimes 3}$ has four irreducible components: one of dimension $5 + 3\binom{3}{2} = 14$, two of dimension $7 + 3\binom{3}{2} = 16$, and one of dimension $9 + 3\binom{3}{2} = 18$.
\end{example}

\Cref{thm:isotopy-invariants} provides an assignment of an irreducible variety of bato tensors to each partial Latin hyperrectangle that is a complete isotopy invariant. Studying the interplay between the algebraic invariants of these varieties and the combinatorial invariants of isotopy classes deserves further consideration.

\section{The world is bato} \label{sec:world}

Generic tensors of order at least three are not bato, yet many tensors arising naturally in algebra, geometry, and complexity theory are.

\begin{example}[Determinant]
    The determinant 
    \[
    \det\nolimits_n = \sum_{\sigma \in \fS_n} \operatorname{sgn}(\sigma) \, \be_{\sigma(1)} \otimes \cdots \otimes \be_{\sigma(n)} \in (\R^n)^{\otimes n}
    \]
    is a bato tensor since two distinct permutations differ in at least two positions. Moreover, its Frobenius norm is $\| \det_n\| = \sqrt{n!}$ and its spectral norm is $\| \det_n\|_\sigma = 1$ by Hadamard's inequality, so $\batorank(\det\nolimits_n) = n!$ by \Cref{prop:bato-rank-lower-bound}.
\end{example}

The remaining examples arise as structure tensors of finite-dimensional real inner-product algebras.
Let $A$ be such an algebra. Its multiplication map
\[
m_A:A\times A\longrightarrow A,
\qquad
(x,y)\longmapsto xy,
\]
is represented by a tensor
\(
\cT_A\in A^*\otimes A^*\otimes A,
\)
called the \emph{structure tensor} of $A$. 
If $\{\be_1,\ldots,\be_n\}$ is an orthonormal basis of $A$ and
\(
\be_i\be_j=\sum_{k=1}^n c_{ij}^k\be_k,
\)
then
\[
\cT_A
=
\sum_{i,j,k=1}^n
c_{ij}^k\,\be_i^*\otimes\be_j^*\otimes\be_k.
\]

\begin{definition}
An inner-product algebra $A$ is called \emph{bato} if its structure
tensor $\cT_A$ is bato.
\end{definition}

From the viewpoint of bilinear complexity, an ordinary rank
decomposition gives an algorithm for multiplication using scalar
multiplications of linear forms. A bato decomposition is a more
structured algorithm: after orthogonal changes of coordinates in the
two inputs and the output, each multiplication gate multiplies one
input coordinate by one input coordinate and routes the result to a
single output coordinate. The partial Latin condition makes this
routing collision-free along every row and column. The statements we make in the following regarding the bato rank of structure tensors follow from the matrix-rank formula for bato slices given in \Cref{prop:bato-rank-slices}.

\begin{example}[Entrywise multiplication]
Equip $A=\R^n$ with the entrywise, or Hadamard, product
\(
(x_1,\ldots,x_n)\odot(y_1,\ldots,y_n)
=
(x_1y_1,\ldots,x_ny_n).
\)
The structure tensor of this algebra is
\[
\cT_{\odot}
=
\sum_{i=1}^n
\be_i^*\otimes\be_i^*\otimes\be_i.
\]
Thus $\cT_{\odot}$ is odeco and $\batorank(\cT_\odot)=n$.
\end{example}

\begin{example}[Matrix multiplication]
For $A = \R^{n \times n}$, the matrix multiplication tensor
\[
\matmul_{n} = \sum_{i,j,k=1}^n E_{ik}^* \otimes E_{kj}^* \otimes E_{ij} 
\]
is bato (here, $\{E_{ij}\}$ denotes the canonical basis). For example, for $n=2$, labeling the tuples $(1,1) \leftrightarrow 1, (1,2) \leftrightarrow 2, (2,1) \leftrightarrow 3, (2,2) \leftrightarrow 4,$ we get the partial Latin square
    \begin{center}
        \begin{tabular}{|c|c|c|c|}
        \hline
        1 & 2 &  &  \\
        \hline
         &  & 1 & 2 \\
        \hline
        3 & 4 &  &  \\
        \hline
        & & 3 & 4 \\
        \hline
        \end{tabular}
    \end{center}
    The bato decomposition of $\matmul_n$ given above has $n^3$ terms. Given $X \in \R^{n \times n} \setminus \{0\}$, left multiplication $L_X : \R^{n \times n} \to \R^{n \times n}, Y \mapsto XY$ satisfies that $\rank(L_X) = n \,\rank(X) \geq n$, so $\batorank(\matmul_n) \geq n^2 \cdot n =  n^3$ by \Cref{prop:bato-rank-slices}. The displayed decomposition attains this bound, so $\batorank(\matmul_n) = n^3$.
\end{example}

\begin{example}[Normed division algebras]
    The real algebras of real numbers $\R$, complex numbers $\C$, quaternions $\mathbb{H}$, and octonions $\mathbb{O}$ are bato. In the standard bases, the product of two units is, up to sign, another unit, and the multiplication tables have Latin structure:\\[1em]
\resizebox{\textwidth}{!}{
\begin{tabular}{@{}c@{\qquad}c@{\qquad}c@{\qquad}c@{}}
$\begin{array}[t]{c|c}
\mathbb{R} & 1\\ \hline
1 & 1
\end{array}$
&
$\begin{array}[t]{c|cc}
\mathbb{C} & 1 & i\\ \hline
1 & 1 & i\\
i & i & -1
\end{array}$
&
$\begin{array}[t]{c|rrrr}
\mathbb{H} & 1 & i & j & k\\ \hline
1 & 1 & i & j & k\\
i & i & -1 & k & -j\\
j & j & -k & -1 & i\\
k & k & j & -i & -1
\end{array}$
&
$\begin{array}[t]{c|rrrrrrrr}
\mathbb{O}
 & 1 & i & j & k & \ell & i\ell & j\ell & k\ell\\ \hline
1
 & 1 & i & j & k & \ell & i\ell & j\ell & k\ell\\
i
 & i & -1 & k & -j & i\ell & -\ell & -k\ell & j\ell\\
j
 & j & -k & -1 & i & j\ell & k\ell & -\ell & -i\ell\\
k
 & k & j & -i & -1 & k\ell & -j\ell & i\ell & -\ell\\
\ell
 & \ell & -i\ell & -j\ell & -k\ell & -1 & i & j & k\\
i\ell
 & i\ell & \ell & -k\ell & j\ell & -i & -1 & -k & j\\
j\ell
 & j\ell & k\ell & \ell & -i\ell & -j & k & -1 & -i\\
k\ell
 & k\ell & -j\ell & i\ell & \ell & -k & -j & i & -1
\end{array}$
\end{tabular}
}
\\[1em]
Given $A\in\{\R,\C,\mathbb H,\mathbb O\}$ and
$x\in A\setminus\{0\}$, the linear map
\(
L_x:A\to A,\ y\mapsto xy,
\)
is invertible. Therefore, every bato decomposition of $\cT_A$ has at
least $(\dim_{\R}A)^2$ terms, so the standard Latin multiplication
table attains this bound. Hence
\[
\batorank(\cT_{\R})=1,\quad
\batorank(\cT_{\C})=4,\quad
\batorank(\cT_{\mathbb H})=16,\quad
\batorank(\cT_{\mathbb O})=64.
\]
\end{example}

\begin{example}[Group algebras]
Let $G$ be a finite group and let $\R[G]$ be its real group algebra.
Equip $\R[G]$ with the inner product for which the elements of $G$
form an orthonormal basis. Its multiplication tensor is
\[
\cT_{\R[G]}
=
\sum_{g,h\in G}
g^*\otimes h^*\otimes gh.
\]
The multiplication table of a group is a Latin square: for fixed
$g$, the map $h\mapsto gh$ is a permutation of $G$, and for fixed
$h$, the map $g\mapsto gh$ is also a permutation. Hence $\R[G]$ is
a bato algebra and
\(
\batorank(\cT_{\R[G]})\leq |G|^2.
\)

The group-basis decomposition need not be minimal. For example, for the cyclic group~$C_2$ the structure tensor $\cT_{\R[C_2]}$ is the odeco tensor shown in \eqref{eq:non-unique-bato}, so $\batorank(\cT_{\R[C_2]})=2 < 4$.
\end{example}

\begin{example}[Cross product]
The cross product on $\R^3$ has structure tensor
\[
\cT_{\times}= \be_1^*\otimes\be_2^*\otimes\be_3
-\be_1^*\otimes\be_3^*\otimes\be_2
+\be_2^*\otimes\be_3^*\otimes\be_1-\be_2^*\otimes\be_1^*\otimes\be_3
+\be_3^*\otimes\be_1^*\otimes\be_2
-\be_3^*\otimes\be_2^*\otimes\be_1,
\]
which is a bato tensor supported on the partial Latin square
\begin{center}
    \begin{tabular}{|c|c|c|}
        \hline
         & 3 & 2 \\
        \hline
        3 &  & 1 \\
        \hline
        2 & 1 & \\
        \hline
    \end{tabular}
\end{center}
For every $x\in\R^3\setminus\{0\}$,
the linear map $\R^3\to\R^3,
\
y\mapsto x\times y$
has kernel $\R x$ and hence rank two. It follows that every bato
decomposition has at least $3\cdot 2=6$ terms, so
\(
\batorank(\cT_{\times})=6.
\)
\end{example}

\begin{example}[Exterior algebras]
Equip the exterior algebra
\(
\bigwedge \R^n=\bigoplus_{r=0}^n\bigwedge^r \R^n
\)
with the inner product for which the standard blades
\(
\be_I=\be_{i_1}\wedge\cdots\wedge\be_{i_r},
\
I=\{i_1<\cdots<i_r\}\subseteq[n],
\)
form an orthonormal basis. For $I,J\subseteq[n]$ one has
\[ 
\be_I\wedge\be_J
=
\begin{cases}
\varepsilon(I,J)\,\be_{I\cup J},
    & I\cap J=\varnothing,\\
0,  & I\cap J\neq\varnothing,
\end{cases}
\]
where $\varepsilon(I,J)\in\{\pm1\}$ is the sign required to reorder
the concatenated indices increasingly.
For fixed $I$, the map $J\mapsto I\cup J$ is injective on the sets
disjoint from $I$, and the analogous statement holds after fixing
$J$. Thus the blade basis gives a partial-Latin multiplication table,
so $\bigwedge \R^n$ is a bato algebra. The bato decomposition
\[
\cT_{\bigwedge \R^n} = \sum_{I, J \subseteq [n]} \be_I^* \otimes \be_J^* \otimes (\be_I \wedge \be_J)
\]
has 
$\#\{(I,J):I\cap J=\varnothing\}=3^n$
terms, because each element of $[n]$ may belong to $I$, to $J$, or
to neither. Consequently,
$
\batorank\!\left(\cT_{\bigwedge \R^n}\right)\leq 3^n.
$

More generally, a Clifford algebra generated by an $n$-dimensional real vector space is bato, and the bato rank of the corresponding structure tensor is at most $(2^n)^2 = 4^n$.
\end{example}

\begin{example}[Truncated polynomials]
Let
\(
A_n=\R[\varepsilon]/(\varepsilon^n)
\)
and equip it with the  orthonormal basis $\bar 1,\bar\varepsilon,\ldots,\bar \varepsilon^{n-1}$. Since
\[
\bar\varepsilon^i \bar\varepsilon^j=
\begin{cases}
\bar \varepsilon^{i+j},&i+j<n,\\
0,&i+j\geq n,
\end{cases}
\]
the monomial basis gives a bato decomposition with
$n(n+1)/2$ terms and this decomposition is minimal. Indeed, for $f\in A_n \setminus\{0\}$, let $\nu(f)$ denote the
$\bar \varepsilon$-adic valuation of $f$ and consider the linear map $L_f : A_n \to A_n, g \mapsto fg$. Then,
$\rank (L_f)=n-\nu(f)$. Moreover, for every basis
$f_1,\ldots,f_n$ of $A_n$, after reordering,
$\nu(f_i)\leq i-1$. Hence
\[
\sum_{i=1}^n\rank L_{f_i}
=
n^2-\sum_{i=1}^n\nu(f_i)
\geq
n^2-\binom n2
=
\binom{n+1}{2}.
\]
Therefore,
\[
\batorank(\cT_{A_n})=\binom{n+1}{2}.
\]
In particular, the structure tensor of the dual numbers $\R[\varepsilon]/(\varepsilon^2)$
has bato rank three.
\end{example}

\begin{example}[Path algebras]
Let $\Gamma$ be a finite acyclic quiver. Its path algebra $\R \Gamma$ has a
basis consisting of all oriented paths in $\Gamma$, including the trivial
path $\varepsilon_v$ at each vertex $v$. Since $\Gamma$ is acyclic, there
are only finitely many such paths. For two basis paths $p$ and $q$, their product is
\[
pq=
\begin{cases}
\text{the concatenated path }pq,
    & \text{if the terminal vertex of $p$ is the initial vertex of $q$},\\
0,  & \text{otherwise}.
\end{cases}
\]
If $p$ is fixed, two different composable paths $q$ give two different
concatenations $pq$; similarly, if $q$ is fixed, two different paths
$p$ give different concatenations. Thus, the path basis gives a
partial-Latin multiplication table, so $\R \Gamma$ is a bato algebra with
\[
\batorank(\cT_{\R \Gamma})
\leq
\sum\nolimits_{v}
\#\{p:t(p)=v\}\,
\#\{q:s(q)=v\},
\]
where $s(q)$ and $t(p)$ denote the initial and terminal
vertices.
\end{example}

\begin{example}[Tensor product of bato algebras]
    The tensor product of two bato algebras $A$ and $B$ is a bato algebra, since the structure tensor of $A \otimes_\R B$ is $\cT_{A \otimes_\R B} = \cT_{A} \boxtimes \cT_{B}$. So we may obtain many other interesting examples by tensoring the algebras introduced above, such as some algebras of hypercomplex numbers.
    However, it is worth recalling that bato rank is submultiplicative under Kronecker products, as shown in \Cref{ex:kronecker-submultiplicative}.
\end{example}

\section*{Acknowledgments}
I would like to thank my advisor Anna Seigal for helpful discussions. 
Thanks also to Paula Esquerra for helping me enumerate the isotopy classes of small partial Latin squares.
I received funding from “la Caixa” Foundation (ID 100010434), under the agreement LCF/BQ/EU23/12010097, and from Real Colegio Complutense at Harvard University.

\bibliographystyle{abbrvnat}
\bibliography{references}

\newpage
\appendix
\section{A census on small maximal partial Latin hyperrectangles}\label{sec:small-latin}

The data in \Cref{tab:small-bato-cardinalities} were obtained by exhaustive search over partial Latin hyperrectangles. For each format, we tested maximality under inclusion and grouped the resulting maximal objects into isotopy classes under independent permutations of the
indices in each coordinate.

\begin{table}[htbp]
\centering
\resizebox{\textwidth}{!}{
\begin{tabular}{c|c|c}
$n_1, \dots, n_d$ &
\# classes
& Cardinalities of each class \\
\hline
  \(2, 2, 2\)
    & \(2\)
    & \(2, 4\) \\
      \(2, 2, 3\)
    & \(2\)
    & \(4^{\times 2}\) \\
      \(2, 2, 4\)
    & \(3\)
    & \(4^{\times 3}\) \\
    \(2, 2, 5\)
    & \(3\)
    & \(4^{\times 3}\) \\
      \(2, 3, 3\)
    & \(2\)
    & \(5, 6\) \\
      \(2, 3, 4\)
    & \(3\)
    & \(6^{\times 3}\) \\
    \(2, 3, 5\)
    & \(4\)
    & \(6^{\times 4}\) \\
  \(2, 4, 4\)
    & \(4\)
    & \(6, 7, 8^{\times 2}\) \\
  \(2, 4, 5\)
    & \(5\)
    & \(8^{\times 5}\) \\
    \(3, 3, 3\)
    & \(4\)
    & \(5, 7^{\times 2}, 9\) \\
    \(3, 3, 4\)
    & \(6\)
    & \(7, 8, 9^{\times 4}\) \\
  \(3, 3, 5\)
    & \(13\)
    & \(9^{\times 13}\) \\
  \(3, 4, 4\)
    & \(12\)
    & \(8^{\times 2}, 9^{\times 3}, 10^{\times 4}, 11, 12^{\times 2}\) \\
  \(3, 4, 5\)
    & \(21\)
    & \(10^{\times 2}, 11^{\times 7}, 12^{\times 12}\) \\
  \(3, 5, 5\)
    & \(41\)
    & \(11^{\times 2}, 12^{\times 9}, 13^{\times 21}, 14^{\times 6}, 15^{\times 3}\) \\
  \(4, 4, 4\)
    & \(49\)
    & \(8, 10^{\times 3}, 12^{\times 32}, 13^{\times 6}, 14^{\times 5}, 16^{\times 2}\) \\
  \(4, 4, 5\)
    & \(122\)
    & \(12^{\times 13}, 13^{\times 10}, 14^{\times 55}, 15^{\times 22}, 16^{\times 22}\) \\
  \(4, 5, 5\)
    & \(600\)
    & \(13^{\times 3}, 14^{\times 4}, 15^{\times 84}, 16^{\times 217}, 17^{\times 208}, 18^{\times 69}, 19^{\times 12}, 20^{\times 3}\) \\
  \(5, 5, 5\)
    & \(6490\)
    & \(13, 15^{\times 5}, 17^{\times 303}, 18^{\times 594}, 19^{\times 2806}, 20^{\times 1628}, 21^{\times 1028}, 22^{\times 92}, 23^{\times 31}, 25^{\times 2}\) \\
    \hline
  \(2, 2, 2, 2\)
    & \(8\)
    & \(4^{\times 6}, 5, 8\) \\
    \(2, 2, 2, 3\)
    & \(10\)
    & \(6^{\times 4}, 8^{\times 6}\) \\
  \(2, 2, 2, 4\)
    & \(29\)
    & \(8^{\times 29}\) \\
  \(2, 2, 2, 5\)
    & \(56\)
    & \(8^{\times 56}\) \\
  \(2, 2, 3, 3\)
    & \(23\)
    & \(6, 7^{\times 2}, 8^{\times 2}, 9, 10^{\times 14}, 11, 12^{\times 2}\) \\
  \(2, 2, 3, 4\)
    & \(81\)
    & \(8^{\times 2}, 10^{\times 9}, 11^{\times 22}, 12^{\times 48}\) \\
  \(2, 2, 3, 5\)
    & \(455\)
    & \(12^{\times 455}\) \\
  \(2, 2, 4, 4\)
    & \(346\)
    & \(8, 11^{\times 2}, 12^{\times 104}, 13^{\times 93}, 14^{\times 106}, 15^{\times 22}, 16^{\times 18}\) \\
  \(2, 2, 4, 5\)
    & \(1512\)
    & \(12^{\times 6}, 13^{\times 8}, 14^{\times 161}, 15^{\times 462}, 16^{\times 875}\) \\
  \(2, 3, 3, 3\)
    & \(200\)
    & \(9^{\times 2}, 10^{\times 9}, 11^{\times 19}, 12^{\times 55}, 13^{\times 30}, 14^{\times 65}, 15^{\times 18}, 16, 18\) \\
  \(2, 3, 3, 4\)
    & \(1698\)
    & \(12^{\times 4}, 13^{\times 29}, 14^{\times 352}, 15^{\times 378}, 16^{\times 571}, 17^{\times 260}, 18^{\times 104}\) \\
  \(2, 3, 3, 5\)
    & \(22668\)
    & \(14^{\times 22}, 15^{\times 113}, 16^{\times 1192}, 17^{\times 6097}, 18^{\times 15244}\) \\
  \(2, 3, 4, 4\)
    & \(33333\)
    & \(14^{\times 11}, 15^{\times 11}, 16^{\times 391}, 17^{\times 1895}, 18^{\times 9560}, 19^{\times 11067}, 20^{\times 7280}, 21^{\times 2440}, 22^{\times 607}, 23^{\times 52}, 24^{\times 19}\) \\
      \(3, 3, 3, 3\)
    & \(10246\)
    & \(9, 12^{\times 9}, 14^{\times 102}, 15^{\times 663}, 16^{\times 1110}, 17^{\times 2754}, 18^{\times 3084}, 19^{\times 999}, 20^{\times 1137}, 21^{\times 355}, 22^{\times 24}, 23^{\times 6}, 24, 27\) \\
    \hline
  \(2, 2, 2, 2, 2\)
    & \(98\)
    & \(8^{\times 75}, 9^{\times 10}, 10^{\times 11}, 12, 16\) \\
  \(2, 2, 2, 2, 3\)
    & \(366\)
    & \(8^{\times 6}, 10^{\times 22}, 11^{\times 12}, 12^{\times 166}, 13^{\times 84}, 14^{\times 24}, 16^{\times 52}\) \\
  \(2, 2, 2, 2, 4\)
    & \(4996\)
    & \(8, 12^{\times 42}, 13^{\times 48}, 14^{\times 576}, 15^{\times 480}, 16^{\times 3849}\) \\
  \(2, 2, 2, 2, 5\)
    & \(99853\)
    & \(16^{\times 99853}\) \\
  \(2, 2, 2, 3, 3\)
    & \(9632\)
    & \(12^{\times 24}, 13^{\times 100}, 14^{\times 767}, 15^{\times 957}, 16^{\times 1815}, 17^{\times 1907}, 18^{\times 1199}, 19^{\times 1293}, 20^{\times 1336}, 21^{\times 184}, 22^{\times 39}, 23^{\times 5}, 24^{\times 6}\) \\
\end{tabular}
}
\vspace{1em}
\caption{Number of maximal isotopy classes of $\Latin(n_1, \dots, n_d)$ and their respective cardinalities. Superscripts indicate multiplicities.}
\label{tab:small-bato-cardinalities}
\end{table}

\end{document}